\documentclass[12pt]{article}
\usepackage{ayv_paperstyle}

\newcommand{\mext}[2]{(#1\triangleright_{#2})}
\newcommand{\taverage}[2]{\lfloor{#1}\rfloor_{#2}}
\newcommand{\extK}{\mathcal{K}\cup \{\star\}}
\newcommand{\extRd}{\rd\cup \{\star\}}
\newcommand{\extS}{\mathcal{S}\cup \{\star\}}
\newcommand{\extSr}{\mathcal{S}^a\cup \{\star\}}

\title{Nonlocal balance equation:\\ representation and approximation of solution}

\author{Yurii Averboukh}
\date{}

\AtEndDocument{
	\vspace{10pt}
	\begin{tabular}{ll}
		Yurii Averboukh:  &
		Krasovskii Institute of Mathematics and Mechanics,  \\ & 
		16 S. Kovalevskaya str.,  Yekaterinburg, Russia; \\
		&
		HSE University,  \\ & 
		11 Pokrovsky Bulvar,  Moscow, Russia;
		\\ & 
		\email{ayv@imm.uran.ru}
\end{tabular}}

\begin{document}
	
	\maketitle
	
\begin{abstract}{We study a nonlocal balance equation that describes the evolution of a  system consisting of infinitely many identical particles those move along a deterministic dynamics and  can also either disappear or give a spring. In this case, the solution of the balance equation is considered in the space of nonnegative measures. We prove the superposition principle for the examined nonlocal balance equation. Furthermore, we interpret the source/sink term as a probability rate of jumps from/to a remote point. Using this idea and replacing the deterministic dynamics of each particle by a nonlinear Markov chain, we approximate the solution of the balance equation  by a solution of a system of ODEs and evaluate the corresponding approximation rate. This result can be used for construction of numerical solutions of the nonlocal balance equation.}
\msccode{35R06, 70F45, 60J27}
\end{abstract}

\section{Introduction}\label{sect:intro}	
The paper is concerned with the study of a nonlocal balance equation, where it is assumed that the source/sink term describes the probability rate for a particle to  disappear or give a spring. Therefore, the examined equation takes the  form:
\begin{equation}\label{intro:eq:nonlocal_continuity}
	\partial_t m(t)+\operatorname{div}(f(t,x,m(t))m(t))=g(t,x,m(t))m(t).
\end{equation} It is endowed with the initial condition 
\[m(0)=m_0.\]
Here, $m(t)$ is a measure on the phase space that describes a distribution of particles. Within the framework of particle interpretation of the nonlocal balance equation,  for each time $t$ and measure $m$, $x\mapsto f(t,x,m)$ is a velocity field, while $g(t,x,m)$ provides a probability rate for a particle to disappear or to give a spring at the point $x$ while the distribution of all particles at time $t$ is $m$. Such form of the nonlocal balance equation is, in particular, used within the study of model of opinion dynamics with time-varying weights \cite{Duteil2022,Piccoli2021}. 

The nonlocal balance equation is a natural extension of the nonlocal continuity equation. The latter was examined in several works (see~\cite{Crippa_existence,Kolokoltsov_DE,KEIMER2023183} and reference therein) and finds various applications in the analysis of pedestrian flows, opinion dynamics, etc~\cite{Ayi2021,Piccoli2018,McQuade2019,Albi2019}. Notice that  balance equation~\eqref{intro:eq:nonlocal_continuity} becomes a continuity equation if we simply put $g\equiv 0$. In this case, the total quantity of elements of the system  is preserved, i.e., the solution is a flow in the space of probability measures. 
The general balance equation implies that the right-hand side is an arbitrary signed measure. Its solution is considered on the space of signed measures. This is due to the fact that, if the negative part of this right-hand side is singular w.r.t. the measure variable, the solution has also a negative part~\cite[Theorem 7.1.1]{Kolokoltsov_DE}. The existence and uniqueness results for general balance equation were discussed in~\cite{Piccoli2014,Piccoli2019,Pogodaev2022}. Additionally, papers \cite{Pogodaev2022,Colombo2015} provide the sensitivity analysis of solution to the general balance equation. Furthermore, let us mention papers \cite{Keimer2018,Bayen2022} where  systems of balance laws were examined.

The main results of the paper are the following. First, we derive the superposition principle that represents a solution of balance equation~\eqref{intro:eq:nonlocal_continuity} through a distribution on the space of curves with time-varying weights (see Theorem \ref{th:superposition}).  The second result is the conservation form of the balance equation. Here we rewrite the balance equation as an equation on a flow of probability measures with the right-hand side given by a L\'evy-Khinchine generator acting on an augmented space (see Theorem \ref{th:restriction_kynetic}). The third result is an approximation of the solution of the original nonlocal balance equation by a solution of a system of ODEs. Here, we assume that the velocity filed $f$ and the initial measure $m_0$ are supported on a compact set $\mathcal{K}$. The approximating system of ODEs takes the form
\begin{equation}\label{intro:eq:ODE}\frac{d}{dt}\beta(t)=\beta(t)Q(t,\beta(t))+\beta(t)G(t,\beta(t)).\end{equation} Here, 
\begin{itemize}
	\item given $t\in [0,T]$, $\beta(t)$ is a row-vector indexed by elements of some finite set $\mathcal{S}$, 
	\item for each time $t$ and vector $\beta$, $G(t,\beta)$ are $Q(t,\beta)$ are matrices indexed alsp by elements of $\mathcal{S}$; 
	\item the matrix $G(t,\beta)$ is diagonal;
	\item the matrix $Q(t,\beta)$ is Kolmogorov (we call a matrix Kolmogorov if it has nonnegative entries outside the diagonal whereas the sum of entries on each row is zero).  
\end{itemize}
In the case where the set $\mathcal{S}$ is close to the compact $\mathcal{K}$ and  the matrix $Q$ approximates the  velocity field on the set $\mathcal{K}$  (see condition \ref{cond:torus}--\ref{cond:variance}),  we   evaluate in Theorem~\ref{th:approx} the distance between the measure $m(t)$ and the empirical distribution $\sum_{{x}\in\mathcal{S}}\beta_{{x}}(t)\delta_{{x}}$ for each time $t$ (here $\delta_{{x}}$ stands for the Dirac measure concentrated at $x$).  Additionally, we give the way how one can construct a set $\mathcal S$ and a matrix $Q$ to derive an arbitrary small approximation rate (see Proposition~\ref{prop:example}).  The matrix $G$ is entirely determined by the function $g$ and the set~$\mathcal{S}$.  Notice that, applying a  numerical scheme to  system of ODEs~\eqref{intro:eq:ODE}, one can obtain  a numerical solution of  nonlocal balance equation~\eqref{intro:eq:nonlocal_continuity} with the controlled accuracy.

Let us briefly comment the main results. 

The superposition principle obtained in the paper is a counterpart of the famous superposition principle for the continuity equation~\cite{Ambrosio,Ambrosio2008b,Stepanov2017,Bonicatto2018,Ambrosio_Crippa}. The latter states the equivalence between the Lagrangian and Eulerian descriptions of systems of identical particles on an Euclidean space. Notice that for the linear balance equation the superposition principle was obtained previously in~\cite{Bredies2022,Maniglia2007}.  The key ingredient in the superposition principle obtained in the paper for  nonlocal balance equation~\eqref{intro:eq:nonlocal_continuity} that is nonlinear is an analog of the evaluation operator that assigns to a measure on the space of trajectories with varying weights a measure on the phase space describing the instantaneous distribution of particles. The latter takes into account not only a position occupied by a particle but also its weight.

Further, recall that the natural phase space for the nonlocal balance equation is the space of nonnegative measures. We endow this space with the distance first proposed by Piccoli and Rossi (see~\cite{Piccoli2014,Piccoli2016,Piccoli2019,Piccoli2018}) that also works for signed measures. It generalizes the famous Wasserstein distance.
Thus, we call this metric  a PRW-distance (PRW-metric). 
It turns out  that   the PRW-distance between two nonnegative measures on a Polish space can be computed using the standard Wasserstein distance between probability measures obtained by adding a mass to a  remote point and normalization (see Section~\ref{sect:preliminaries} for the details).

The concept of remote point gives an interpretation of the source/sink term. An appearance of a particle is viewed as a jump from this remote point, while a disappearance is a jump to the remote point. Writing down operators those describe these jump processes, we arrive at the conservation form of the nonlocal balance equation.

Due to fact that the  $Q$ is a Kolmogorov matrix,  system~\eqref{intro:eq:ODE} describes an evolution of a distribution in  an infinite particle system where each particle can jump on the set $\mathcal{S}$ or disappear/give a spring. Adding to the set $\mathcal{S}$  the remote point, we derive the conservation form of~\eqref{intro:eq:ODE} that determines  a continuous-time Markov chain acting on the set $\mathcal{S}$ augmented by the remote point (see Proposition~\ref{prop:generator_finite}). Here, we, as above, view a disappearance of a particle as a jump to the remote point, while an appearance of a particle is regarded as a jump from the remote point. Notice that the approximation of the deterministic evolution with a controlled accuracy by a Markov chain was previously developed for mean field type systems (in this case, the evolution is given by a continuity equation) in~\cite{Averboukh2021} and for zero-sum differential games in~\cite{Averboukh2016}. The main novelty of the  case examined in the paper is the presence of jumps from/to the remote point with the intensities depending on current distributions. Thus, to evaluate the approximation rate, we propose a  synchronization of these jump processes in the original and approximating systems.



The paper is organized as follows. In Section~\ref{sect:preliminaries}, we give general notation and recall the notion of PRW-distance proposed in~\cite{Piccoli2014,Piccoli2019}. Moreover, this Section provides an equivalent form of the PRW-distance as a normalized usual Wasserstein distance between probabilities on a space augmented by a remote point. The equilibrium distribution of the curves with weights is discussed in Section~\ref{sect:weights}. Here, we, in particular, prove the existence and uniqueness result for the equilibrium distribution of curves with weights. In the next section, we derive the superposition principle for the examined nonlocal balance equation. Further, we rewrite the balance equation in the conservation form (see Section~\ref{sect:particle_balance}). In Section~\ref{sect:finite}, we introduce an approximating system of ODEs. Finally, we evaluate the approximation rate in Section~\ref{sect:coupled}.


\section{Preliminaries}\label{sect:preliminaries}
\subsection{General notation}
Given $a,b\in\mathbb{R}$, we denote by $a\wedge b$ the minimum of these numbers. Similarly, $a\vee b$ states for the maximum of $a$ and $b$.

If $n$ is a natural number, $X_1,\ldots,X_n$ are sets, $i_1,\ldots,i_k$ are distinct numbers from $\{1,\ldots,n\}$, then we denote by $\operatorname{p}^{i_1,\ldots,i_k}$ the natural projection from $X_1\times\ldots X_n$ onto $X_{i_1}\times\ldots X_{i_k}$, i.e.,
\[\operatorname{p}^{i_1,\ldots,i_k}(x_1,\ldots,x_n)\triangleq (x_{i_1},\ldots,x_{i_k}).\]

If $(\Omega,\mathcal{F})$ is a measurable space, $m$ is a measure on $\Omega$, then we denote by $\|m\|$ the total mass of $m$, i.e., $\|m\|\triangleq m(\Omega)$.  Furthermore, if $m_1,m_2$ are measures on $\mathcal{F}$, then we will write that $m_1\leqq m_2$ in the case where, for each $\Upsilon\in\mathcal{F}$,
\[m_1(\Upsilon)\leq m_2(\Upsilon).\]

If $(\Omega,\mathcal{F})$, $(\Omega',\mathcal{F}')$ are two measurable spaces, $m$ is a measure on $\mathcal{F}$, $h:\Omega\rightarrow\Omega'$ is a $\mathcal{F}/\mathcal{F}'$-measurable mapping, then  $h\sharp m$ stands for the push-forward of the measure $m$  through $h$ defined by the rule: for each $\Upsilon\in\mathcal{F}'$,
\[(h\sharp m)(\Upsilon)\triangleq m(h^{-1}(\Upsilon)).\] 

If $(X,\rho_X)$, $(Y,\rho_Y)$ are Polish spaces, then 
$C(X,Y)$ denotes the set of all continuous functions from $X$ to $Y$; $C_b(X,Y)$ is the subset of bounded continuous  functions. Further, $\operatorname{Lip}_{c}(X,Y)$  stands for the set of functions from $X$ to~$Y$ those are  Lipschitz continuous with the constant equal to $c$. When $Y=\mathbb{R}$, we will omit the second argument. In particular, $C_b(X)$ denotes the set of all continuous bounded functions from $X$ to $\mathbb{R}$. Moreover, $C_b(X)$ is endowed with the usual $\sup$-norm.

If $(X,\rho_X)$ is a Polish space, then we denote by $\mathcal{M}(X)$ the set of all (nonnegative) finite Borel measures on $X$, by $\mathcal{P}(X)$ the set of all probabilities on $X$, i.e.,
\[\mathcal{P}(X)=\{m\in\mathcal{M}(X):m(X)=1\}.\] 

On $\mathcal{M}(X)$, we consider the narrow convergence that is defined as follows: we say that a sequence of measures $\{m_n\}_{n=1}^\infty$ narrowly convergence to $m\in\mathcal{M}(X)$, if, for every $\phi\in C_b(X)$,
\[\int_X \phi(x)m_n(dx)\rightarrow \int_X\phi(x)m(dx)\text{ as }n\rightarrow \infty.\]  Notice that $\mathcal{P}(X)$ is closed in the topology of narrow convergence.


Further, if $p\geq 1$, then we consider the set of all probabilities with the finite $p$-th moment $\mathcal{P}^p(X)$. This means that a probability $m$ lies in $\mathcal{P}^p(X)$ iff, for some (equivalently, every) $x_*\in X$,
\[\int_X (\rho_X(x,x_*))^pm(dx)<\infty.\] 

The space $\mathcal{P}^p(X)$ is endowed with the Wasserstein distance defined by the rule: for $m_1,m_2\in\mathcal{P}^p(X)$,
\[W_p(m_1,m_2)\triangleq \Bigg[\inf_{\pi\in \Pi(m_1,m_2)}\int_{X\times X}(\rho_X(x_1,x_2))^p\pi(d(x_1,x_2))\Bigg]^{1/p}.\] Here $\Pi(m_1,m_2)$ denotes the set of  plans between $m_1,m_2$ that consists of all probabilities on $X\times X$ such that $\operatorname{p}^1\sharp \pi=m_1$, $\operatorname{p}^2\sharp \pi=m_2$.

\subsection{PRW-distance}\label{subsect:general_Wasserstein}
In this section, we discuss the extension of the Wasserstein distance to the space of nonnegative measure. The definition used in the paper follows the approaches to this problem proposed in~\cite{Piccoli2014,Piccoli2016,Piccoli2019}. 
\begin{definition}\label{def:generalized_wasserstein} Let $m_1,m_2\in \mathcal{M}(X)$, and let $p\geq 1$, $b>0$. We call the quantity $\mathcal{W}_{p,b}(m_1,m_2)$ defined by the rule
	\begin{equation}\label{prel:equality:W}\begin{split}
			(\mathcal{W}_{p,b}(m_1,m_2))^p\triangleq \\\inf\Bigg\{b^p\|m_1-\widehat{m}&{}_1\|+b^p\|m_2-\widehat{m}_2\|+\int_{X\times X}(\rho_X(x_1,x_2))^p\varpi(d(x_1,x_2)):\\ &{}\hspace{20pt}\widehat{m}_1\leqq m_1,\, \widehat{m}_2\leqq m_2,\, \|\widehat{m}_1\|=\|\widehat{m}_2\|,\, \varpi\in \Pi(\widehat{m}_1,\widehat{m}_2)\Bigg\}\end{split}\end{equation} a PRW-distance. 
	
	Here, for $m_1,m_2\in\mathcal{M}(X)$ of the equal masses ($\|m_1\|=\|m_2\|$),  $\Pi(m_1,m_2)$ is the set of measures on $X\times X$ such that $\operatorname{p}^1\sharp \pi=m_1$, $\operatorname{p}^2\sharp \pi=m_2$. 
\end{definition}
Notice that $\mathcal{W}_{p,b}$ coincides up to renormalization and renaming of parameters with the generalized Wasserstein distance introduced in \cite[Definition 11]{Piccoli2019}.

In the paper, we will widely use the representation of the metric $\mathcal{W}_{p,b}$ as a usual Wassertein metric on an augmented space. To construct this representation, we, first, extend a Polish space $(X,\rho_X)$ by adding to $X$ a remote point $\star$. The distance on $X\cup\{\star\}$ is defined by the rule:
\begin{equation}\label{prel:intro:rho_star}\rho_\star(x,y)\triangleq \left\{
	\begin{array}{ll}
		b, & x\in X, y=\star\text{ or }x=\star,y\in X,\\
		0, & x=y=\star,\\
		\rho_X(x,y)\wedge 2b, & x,y\in X
	\end{array}
	\right.\end{equation} Here $b$ is a positive constant used in Definition \ref{def:generalized_wasserstein}. When it does not lead to a confusion, we will omit the subindex $\star$. In particular, if $\mathcal{K}$ is a compact subset of the finite-dimensional space, $b>\operatorname{diam}(\mathcal{K})/2$, the restriction of the distance $\rho_\star$ on $\mathcal{K}$ coincides with the usual Euclidean distance. Thus, we will denote the distance on $\extK$ by $\|x_1-x_2\|$. 

If $m$ is a measure on $X$,  $R> \|m\|$, then we can extend $m$ to the space $X\cup\{\star\}$ setting, for a Borel set $\Upsilon\subset X\cup\{\star\}$,
\begin{equation}\label{intro:intro:m_ext}\mext{m}{R}(\Upsilon)\triangleq m(\Upsilon\cap X)+(R-m(X))\mathbbm{1}_\Upsilon(\star).\end{equation} 
If $m_1,m_2\in \mathcal{M}(X)$, when  $R\geq \|m_1\|\vee \|m_2\|$, then we put 
\begin{equation}\label{intro:W_p_b}\widetilde{\mathcal{W}}_{p,b}(m_1,m_2)\triangleq R^{1/p}W_p(R^{-1}\mext{m_1}{R},R^{-1}\mext{m_2}{R}). \end{equation}


\begin{proposition}\label{prop:W_p_b} If $m_1,m_2\in \mathcal{M}(X)$,  $R> \|m_1\|\vee\|m_2\|$, then \[\widetilde{\mathcal{W}}_{p,b}(m_1,m_2)=\mathcal{W}_{p,b}(m_1,m_2).\] In particular, the quantity $\widetilde{\mathcal{W}}_{p,b}(m_1,m_2)$ does not depend on the choice of  $R$ greater or equal to $\|m_1\|\vee \|m_2\|$.
	
\end{proposition}
\begin{proof}
	We have that 
	\[\begin{split}
		(\widetilde{\mathcal{W}}_{p,b}(m_1,m_2))^p=\inf\Bigg\{\int_{(X\cup\{\star\})\times (X\cup\{\star\})}(\rho_\star(x_1,&x_2))^p\pi(d(x_1,x_2)):\\ &\pi\in \Pi(\mext{m_1}{R},\mext{m_2}{R})\Bigg\}.\end{split}\] Given, $\pi\in \Pi(\mext{m_1}{R},\mext{m_2}{R})$ that is an optimal plan between $\mext{m_1}{R}$ and $\mext{m_2}{R}$, we represent $\pi$ as the sum of four measures:
	\[\pi=\pi_{X,X}+\pi_{X,\star}+\pi_{\star,X}+\pi_{\star,\star},\] where
	\begin{itemize}
		\item $\operatorname{supp}(\pi_{X,X})\subset X\times X$;
		\item $\operatorname{supp}(\pi_{X,\star})\subset X\times \{\star\}$;
		\item $\operatorname{supp}(\pi_{\star,X})\subset \{\star\}\times X$;
		\item $\operatorname{supp}(\pi_{\star,\star})\subset \{(\star,\star)\}$.
	\end{itemize} Notice that 
	\[m_1=\operatorname{p}^1\sharp(\pi_{X,X}+\pi_{X,\star}),\ \ m_2=\operatorname{p}^2\sharp(\pi_{X,X}+\pi_{\star,X}).\]
	Set $\widehat{m}_1\triangleq \operatorname{p}^1\sharp \pi_{X,X}$, $\widehat{m}_2\triangleq \operatorname{p}^2\sharp \pi_{X,X}$. By construction, we have that 
	\[\widehat{m}_1\leqq m_1,\ \ \widehat{m}_2\leqq m_2. \]
	Due to the definition of the measures $\mext{m_1}{R}$, $\mext{m_2}{R}$, we have that 
	\begin{equation}\label{prel:equality:pi_gen_wasserstein}\begin{split}
			\int_{(X\cup\{\star\})\times (X\cup\{\star\})}(&\rho_\star(x_1,x_2))^p\pi(d(x_1,x_2))\\=
			b^p\|m_1-\widehat{m}&{}_2\|+b^p\|m_2-\widehat{m}_2\|+\int_{X\times X}(\rho_\star(x_1,x_2))^p\pi_{X\times X}(d(x_1,x_2)).
		\end{split}
	\end{equation} Further, the assumption that $\pi$ is an optimal plan between $\mext{m_1}{R}$, $\mext{m_2}{R}$ and the definition  of the distance $\rho_\star$ (see~\eqref{prel:intro:rho_star}), give that  $\rho_\star=\rho_X$ on $\operatorname{supp}(\pi_{X,X})$. 
	This and~\eqref{prel:equality:pi_gen_wasserstein} imply that 
	\begin{equation}\label{prel:ineq:W_p_b}
		(\widetilde{\mathcal{W}}_{p,b}(m_1,m_2))^p\geq
		(\mathcal{W}_{p,b}(m_1,m_2))^p.
	\end{equation} To prove the opposite inequality, we choose  three measures $\widehat{m}_1\in\mathcal{M}(X)$, $\widehat{m}_2\in\mathcal{M}(X)$, $\varpi\in\mathcal{M}(X\times X)$ such that 
	\[\|\widehat{m}_1\|=\|\widehat{m}_2\|,\, \widehat{m}_1\leqq m_1,\, \widehat{m}_2\leqq m_2,\ \ \varpi\in\Pi(\widehat{m}_1,\widehat{m}_2).\] Further, let $\mathscr{r}$ be a mapping that assigns to each element $x\in X$ the remote point $\star$. 
	Since $R\geq \|m_1\|\vee\|m_2\|$, the probability $\pi\in \Pi(R^{-1}\mext{m_1}{R},R^{-1}\mext{m_2}{R})$ defined by the rule:
	\[\pi\triangleq R^{-1}\Bigg[\varpi+(\operatorname{Id},\mathscr{r})\sharp (m_1-\widehat{m}_1)+(\mathscr{r},\operatorname{Id})\sharp (m_2-\widehat{m}_2)+(R-(\|m_1\|\vee\|m_2\|))\delta_{(\star,\star)}\Bigg]\] lies in $\Pi(R^{-1}\mext{m_1}{R},R^{-1}\mext{m_2}{R})$. We have that
	\[\begin{split}
		R \int_{(X\cup\{\star\})\times (X\cup\{\star\})}(&\rho_\star(x_1,x_2))^p\pi(d(x_1,x_2))\\=
		\int_{X\times X}&(\rho_X(x_1,x_2))^p\varpi(x_1,x_2)+b^p\|m_1-\widehat{m}_1\|+b^p\|m_2-\widehat{m}_2\|.
	\end{split}
	\] Hence, \begin{equation*}
		(\widetilde{\mathcal{W}}_{p,b}(m_1,m_2))^p\leq (\mathcal{W}_{1,p}(m_1,m_2))^p.
	\end{equation*} This together with~\eqref{prel:ineq:W_p_b} gives the conclusion of the proposition.
\end{proof}

Now let us estimate the total mass of the measure through its distance to the given 
measure.

\begin{proposition}\label{prop:estimate_mass}
	Let $m_1,m_2\in\mathcal{M}(X)$, $\mathcal{W}_{p,b}(m_1,m_2)\leq c$. Then $\|m_1\|\leq b^{-p}c^p+\|m_2\|$.
\end{proposition}
\begin{proof}
	From Definition \ref{def:generalized_wasserstein} and the assumption of the proposition,  we have that 
	\[b^p(\|m_1\|-\|\widehat{m}_1\|)\leq b^p\|m_1-\widehat{m}_1\|\leq c^p,\] where $\hat{m}_1$, $\hat{m}_2$, $\varpi$ is a triple providing the minimum in the right-hand side of~\eqref{prel:equality:W}. Furthermore, 
	\[\|\widehat{m}_1\|=\|\widehat{m}_2\|\leq \|m_2\|.\] This gives the conclusion of the proposition.
\end{proof}

\begin{proposition}\label{prop:complete_compact}
	The space of finite measure on $X$, $\mathcal{M}(X)$, endowed with the distance $\mathcal{W}_{p,b}$ is complete. If, additionally, $X$ is compact, $R>0$, then $\{m\in \mathcal{M}(X):\|m\|\leq R\}$ is also compact.
\end{proposition}
\begin{proof}
	To prove the first statement of the proposition, consider a Cauchy sequence of measures $\{m_i\}_{i=1}^\infty$. From Proposition~\ref{prop:estimate_mass}, it follows that 
	$\|m_i\|\leq R$ for some positive constant $R$. Thus,  we have that the sequence of probabilities on $X\cup\{\star\}$ $\{R^{-1}\mext{m_i}{R}\}_{i=1}^\infty$ is Cauchy in $W_p$. Therefore, $\{R^{-1}\mext{m_i}{R}\}_{i=1}^\infty$ converges to some probability $\widetilde{m}\in\mathcal{P}^p(X\cup\{\star\})$. Letting $m$ be a restriction of $R\widetilde{m}$ on $X$, we conclude that $\{m_i\}_{i=1}^\infty$ converges to $m$ in $\mathcal{W}_{p,b}$.
	
	To prove the second statement of the proposition, it suffices to notice that the set of probabilities on $X\cup\{\star\}$ is compact, while the mapping $ m\mapsto R^{-1}\mext{m}{R}$ is a isomorphism between the sets $\{m\in \mathcal{M}(X):\|m\|\leq R\}$ and $\mathcal{P}^p(X\cup\{\star\})$.
\end{proof}

Below, we restrict our attention to the case when $p=1$. This metric coincides up to multiplicative constant  with one used in~\cite{Piccoli2014}.

\subsection{Phase space and space of weighted curves}

Let $\mathcal{M}^1(\rd)$ be a set of measures $m$ on $\rd$ such that $\int_{\rd}\|x\|m(dx)<\infty$. Additionally, denote by $\mathcal{D}(\rd)$ the set of continuous functions $\phi$ with sublinear growth.

If $T>0$, we denote the set of continuous functions $\gamma:[0,T]\rightarrow \rd\times [0,+\infty)$ by $\Gamma_T$. An element of $\Gamma_T$ will be interpret as a  curve with time-varying weight. Indeed, if $\gamma(\cdot)=(x(\cdot),w(\cdot))\in \Gamma_T$, $t\in [0,T]$, we regard $w(t)$ as a relative weight of particles at the point $x=x(t)$ at time $t$. If $C>0$ then, we denote by $\Gamma^C_T$ the set curves with  weights bounded by the constant $C$, i.e., 
\[\Gamma_T^C\triangleq \big\{\gamma(\cdot)=(x(\cdot),w(\cdot))\in \Gamma_T:\, w(t)\in  [0,C]\text{ for each }t\in [0,T]\big\}.\]

Below, we will consider the probabilities on $\Gamma_T^C$ with a finite first moment. If $\eta\in \mathcal{P}^1(\Gamma_T^C)$, $t\in [0,T]$, then we denote by $\taverage{\eta}{t}$ the measure on $\rd$ defined by the rule: for $\phi\in C_b(\rd)$,
\begin{equation}\label{prel:intro:taverage}\int_{\mathcal{K}}\phi(x)\taverage{\eta}{t}(dx)\triangleq \int_{\Gamma}\phi(x(t))w(t)\eta(d(x(\cdot),w(\cdot))).\end{equation} The measure $\taverage{\eta}{t}$ gives the distribution of the masses at time $t$ on the phase space in the case when the particles and their weights are distributed according to $\eta$. 

Notice that the measure $\taverage{\eta}{t}\in\mathcal{M}^1(\rd)$. Furthermore, equality~\eqref{prel:intro:taverage} holds true for each $\phi\in\mathcal{D}(\rd)$.

Now let us show that the mapping $\mathcal{P}^1(\Gamma_T^C)\ni \eta\mapsto \taverage{\eta}{t}$ is Lipschitz continuous.

\begin{proposition}\label{prop:eta_W} If $C>0$, $\eta_1,\eta_2\in\mathcal{P}^1(\Gamma_T^C)$, $t\in [0,T]$, then 
	\[\mathcal{W}_{1,b}(\taverage{\eta_1}{t},\taverage{\eta_2}{t})\leq C_0W_1(\eta_1,\eta_2),\] where $C_0=C\vee b^{-1}$.
\end{proposition}
\begin{proof}
	Let $m_1\triangleq \taverage{\eta_1}{t}$, $m_2\triangleq \taverage{\eta_2}{t}$, and let $\pi\in \Pi(\eta_1,\eta_2)$ be an optimal plan between $\eta_1$ and $\eta_2$.  We define the measure $\varpi$ by the following rule: for $\phi\in C_b(\rd\times \rd)$,
	\begin{equation}\label{prel:intro:varpi}\begin{split}\int_{\rd\times\rd}&\phi(x_1,x_2)\varpi(d(x_1,x_2))\\\triangleq &\int_{\Gamma_T^C\times\Gamma_T^C}\phi(x_1(t),x_2(t))(w_1(t)\wedge w_2(t))\pi(d((x_1(\cdot),w_1(\cdot)),(x_2(\cdot),w_2(\cdot)))).\end{split}\end{equation} Further, set $\widehat{m}_1\triangleq \operatorname{p}^1\sharp \varpi$, $\widehat{m}_2\triangleq \operatorname{p}^2\sharp \varpi$. Obviously, $\varpi\in \Pi(\widehat{m}_1,\widehat{m}_2)$. Further, for each $\phi\in C_b(\rd)$,
	\begin{equation*}\label{prel:intro:m_1_tilde}\begin{split}\int_{\rd}\phi&(x)(m_1-\widehat{m}_1)(dx) \\&=\int_{\Gamma_T^C\times\Gamma_T^C}\phi(x_1(t))(w_1(t)-(w_1(t)\wedge w_2(t)))\pi(d((x_1(\cdot),w_1(\cdot)),(x_2(\cdot),w_2(\cdot)))),\end{split}\end{equation*} 
	while 
	\begin{equation}\label{prel:intro:m_2_tilde}\begin{split}
			\int_{\rd}\phi(x)(m_2-&\widehat{m}_2)(dx)\\=&\int_{\Gamma_T^C\times\Gamma_T^C}\phi(x_2(t)) (w_2(t)-(w_1(t)\wedge w_2(t)))\\&{}\hspace{60pt}\pi(d((x_1(\cdot),w_1(\cdot)),(x_2(\cdot),w_2(\cdot)))).\end{split}\end{equation} 
	
	Due to the fact that $\eta_1,\eta_2\in\mathcal{P}^1(\Gamma_T^C)$, one can extend equalities~\eqref{prel:intro:varpi}--\eqref{prel:intro:m_2_tilde} to the set of functions $\phi$ with sublinear growth. 
	Therefore,
	\begin{equation}\label{prel:ineq:varpi_image}
		\begin{split}
			\int_{\rd\times\rd}&\|x_1-x_2\|\varpi(d(x_1,x_2))\\ =&\int_{\Gamma_T^C\times\Gamma_T^C}\|x_1(t)-x_2(t)\|(w_1(t)\wedge w_2(t))\pi(d((x_1(\cdot),w_1(\cdot)),(x_2(\cdot),w_2(\cdot)))) \\ \leq& 
			C \int_{\Gamma_T^C\times\Gamma_T^C}\|x_1(t)-x_2(t)\| \pi(d((x_1(\cdot),w_1(\cdot)),(x_2(\cdot),w_2(\cdot)))).
		\end{split}
	\end{equation} Furthermore, using the Jensen's inequality, we deduce that
	\begin{equation*}\label{prel:ineq:m_1_image}
		\begin{split}
			\|m_1&-\widehat{m}_1\|+\|m_2-\widehat{m}_2\|\\
			&=\int_{\Gamma_T^C\times\Gamma_T^C} ((w_1(t)\vee w_2(t))-(w_1(t)\wedge w_2(t)))\pi(d((x_1(\cdot),w_1(\cdot)),(x_2(\cdot),w_2(\cdot))))\\&= 
			\int_{\Gamma_T^C\times\Gamma_T^C} |w_1(t)-w_2(t)|\pi(d((x_1(\cdot),w_1(\cdot)),(x_2(\cdot),w_2(\cdot)))).
		\end{split}	
	\end{equation*} Combining this with~\eqref{prel:ineq:varpi_image}, we obtain the statement of the proposition.
\end{proof}

\begin{proposition}\label{prop:evaluation_lipschitz} Let ${c}$ be a constant such that, for $\eta$-a.e. curves $(x(\cdot),w(\cdot))\in\Gamma_T^C$, one has that 
	\begin{equation}\label{prel:Lip_s_r}\|x(s)-x(r)\|\leq {c}|r-s|,\, |w(s)-w(r)|\leq {c}|r-s|.\end{equation} Then,
	\[\mathcal{W}_{1,b}(\taverage{\eta}{s},\taverage{\eta}{r})\leq {c}(C+b)|r-s|.\] 
\end{proposition}
\begin{proof}
	Without loss of generality, we assume that $s<r$. For each curve $(x(\cdot),w(\cdot))\in\Gamma_T$ satisfying~\eqref{prel:Lip_s_r}, we have that 
	\[\|x(s)(w(s)\wedge w(r))-x(r)(w(s)\wedge w(r))\|\leq C{c}(r-s).\] Furthermore,
	\[|w(s)-w(r)|\leq {c}(r-s).\] Let $\varpi$ be defined by the rule: for each $\phi\in C_b(\rd\times\rd)$,
	\[\int_{\rd\times\rd}\phi(x_1,x_2)\varpi(d(x_1,x_2))=\int_{\Gamma_T^C}\phi(x(s),x(r))(w(s)\wedge w(r))\eta(d(x(\cdot),w(\cdot))),\]
	\[\widehat{m}_s\triangleq \operatorname{p}^1\sharp\varpi,\ \ \widehat{m}_r\triangleq \operatorname{p}^2\sharp\varpi. \] Therefore, we have that 
	\[\begin{split}
		\mathcal{W}_{1,b}(&\taverage{\eta}{s},\taverage{\eta}{r})\\&\leq \int_{\rd\times \rd}\|x_1-x_2\|\varpi(d(x_1,x_2))+b\big\|\taverage{\eta}{s}-\widehat{m}_s\big\| +b\big\|\taverage{\eta}{r}-\widehat{m}_r\big\|\\ &=
		\int_{\Gamma_T^C}\big(\|x(s)-x(r)\|(w(s)\wedge w(r))\\&{}\hspace{70pt}+b((w(s)\vee w(r))-(w(s)\wedge w(r)))\big)\eta(d(x(\cdot),w(\cdot))\\ &\leq {c}(C+b)(r-s).
	\end{split}\]
\end{proof}

\section{Distribution on the space of varying weighted curves}\label{sect:weights}
In this section, we examine the equilibrium distribution of  curves with time-varying weights satisfying the dynamics 
\begin{equation}\label{distr:eq:x_w}
	\begin{split}
		&\frac{d}{dt}x(t)=f(t,x(t),\taverage{\eta}{t}),\\ 
		&\frac{d}{dt}w(t)=g(t,x(t),\taverage{\eta}{t})w(t).
	\end{split}
\end{equation}

The initial distribution is assumed to be the same as for balance equation~\eqref{intro:eq:nonlocal_continuity} and equal to $m_0$.

\begin{definition}\label{def:distribution} Let $C>0$. We say that $\eta\in \mathcal{P}^1(\Gamma_T^C)$ is an equilibrium distribution of the weighted trajectories provided that 
	\begin{itemize}	
		\item	 $\taverage{\eta}{0}=m_0$;
		\item  $\eta$-a.e. curves $(x(\cdot),w(\cdot))$ satisfy~\eqref{distr:eq:x_w} and the initial condition $w(0)=\|m_0\|$.
	\end{itemize}
\end{definition}

In the following, we  impose the following conditions on $f$ and $g$:
\begin{enumerate}[label=(A\arabic*), series=main_cond]
	\item \label{cond:main:firts_moment} $m_0\in \mathcal{M}^1(\rd)$;
	\item\label{cond:main:continuity} $f$ and $g$ are continuous;
	\item\label{cond:main:Lip} $f$ and $g$ are Lipschitz continuous w.r.t. $x$ and $m$ with the Lipschitz constants $C_{Lf}$ and $C_{Lg}$ respectively;
	\item\label{cond:main:boundness} there exists a constants $C_f$ and $C_g$ such that, for every $t\in [0,T]$, $x\in\rd$, $m\in\mathcal{M}(\rd)$,  \[\|f(t,x,m)\|\leq C_f,\ \ |g(t,x,m)|\leq C_g.\]
\end{enumerate} 
\begin{theorem}\label{th:existence}
	For each $T>0$, $C>C_1\triangleq \|m_0\|e^{TC_g}$, there exists a unique equilibrium distribution of weighted trajectories lying in  $\mathcal{P}^1(\Gamma_T^C)$. 
\end{theorem}
\begin{proof}
	First, let us define 
	\[\widetilde{\Gamma}_T\triangleq \Big\{(x(\cdot),w(\cdot))\in \Gamma_T^{C_1}:x(\cdot)\in \operatorname{Lip}_{C_f}([0,T];\rd),\, w(\cdot)\in \operatorname{Lip}_{C_gC_1}([0,T];\mathbb{R})\Big\}.\] 
	
	Furthermore, we denote 
	\[\mathscr{P}\triangleq \Big\{\eta\in \mathcal{P}^1(\widetilde{\Gamma}_T):\ \ e_0\sharp\eta= (\operatorname{Id},\|m_0\|)(\|m_0\|^{-1}m)\Big\}.\] Hereinafter, $e_t$ stands for the evaluation operator form $\Gamma_T$ to $\rd\times [0,+\infty)$ that assigns to each $\gamma\in\Gamma_T$ its value at time $t$, i.e., $e_t(\gamma)=\gamma(t)$. 
	
	The definition of the set $\mathscr{P}$ gives, in particular, that, for $\eta\in\mathscr{P}$ and $\eta$-a.e. $\gamma\in\widetilde{\Gamma}_T$, one has 
	\[\|\gamma\|\leq TC_f+C_1+\|\gamma(0)\|.\] Therefore, probabilities from $\mathscr{P}$ are uniformly integrable. The assumption $\mathscr{P}\subset \mathcal{P}^1(\widetilde{\Gamma}_T)$, definition of $\widetilde{\Gamma}_T$ and the Arzel\`a-Ascoli theorem give that $\mathscr{P}$ is tight. Therefore, we have that $\mathscr{P}$ is compact in the topology produced by the metric $W_1$ \cite[Proposition 7.1.5]{Ambrosio}.

	If $\eta\in \mathscr{P}$, then denote by $\operatorname{traj}_{\eta}(x_0)$ the operator that assigns to $x_0\in \rd$, the pair of functions $(x(\cdot),w(\cdot))$ that satisfies
	\begin{equation*}\label{distr:eq:x_w_0_m_0}
		\begin{split}
			&\frac{d}{dt}x(t)=f(t,x(t),\taverage{\eta}{t}),\ \ x(0)=x_0;\\
			&\frac{d}{dt}w(t)=g(t,x(t),\taverage{\eta}{t})w(t),\ \ w(0)=\|m_0\|.
		\end{split}
	\end{equation*} Further, set
	\begin{equation}\label{distr:intro:Phi}
		\Phi(\eta)\triangleq \operatorname{traj}_{\eta}\sharp (\|m_0\|^{-1}m_0).
	\end{equation} By construction, $\Phi(\eta)\in \mathcal{P}(\Gamma_T)$. Further,
	\begin{equation*}\label{distr:equality:m_0}
		\taverage{\Phi(\eta)}{0}=m_0.
	\end{equation*} 
	
	Notice that $\eta$ is a equilibrium distribution of curves with time-varying weights if and only if
	\[\Phi(\eta)=\eta.\]

	Let us show that the operator $\Phi$ defined by~\eqref{distr:intro:Phi} maps $\mathscr{P}$ into itself. 
	To this end, we choose $(x(\cdot),w(\cdot))=\operatorname{traj}_{\eta}(x_0)$. For each $s,r\in [0,T]$, one has that
	\[ 
	\|x(r)-x(s)\|\leq \int_s^r \|f(t,x(t),\taverage{\eta}{t})\|dt.
	\] This and condition~\ref{cond:main:boundness} yield that $x(\cdot)\in \operatorname{Lip}_{C_f}([0,T];\rd)$.
	Further, since $|g(t,x(t),\taverage{\eta}{t})|\leq C_g$, we, using the Gronwall's inequality, conclude that 
	$|w(t)|\leq C_1$. Additionally, for $s,r\in [0,T]$,
	\[|w(r)-w(s)|\leq \int_s^rw(t)|g(t,x(t),\taverage{\eta}{t})|dt\leq C_gC_1|r-s|.\] Thus, $w(\cdot)\in \operatorname{Lip}_{C_gC_1}([0,T],\mathbb{R})$. Therefore, we have that $\operatorname{traj}_{\eta}(x_0)\in\widetilde{\Gamma}_T$, while by construction
	\[e_0\sharp(\Phi(\eta))= (\operatorname{Id},\|m_0\|)(\|m_0\|^{-1}m).\] This means that $\Phi(\eta)\in \mathscr{P}$.

	Now we prove that $\Phi$ is a continuous mapping on $\mathscr{P}$. Let $\eta_1,\eta_2\in\mathcal{P}(\widetilde{\Gamma}_T)$ and let $(x_1(\cdot),w_1(\cdot))=\operatorname{traj}_{\eta_1}(x_0)$, $(x_2(\cdot),w_2(\cdot))=\operatorname{traj}_{\eta_2}(x_0)$. We have that 
	\[\|x_1(t)-x_2(t)\|\leq \int_{0}^t\|f(t',x_1(t'),\taverage{\eta_1}{t'})-f(t',x_2(t'),\taverage{\eta_2}{t'})\|.\]
	
	Using the Lipschitz continuity and Proposition~\ref{prop:eta_W}, we obtain the estimate
	\[\|x_1(t)-x_2(t)\|\leq \int_{0}^t C_{Lf}\|x_1(t')-x_2(t')\|dt+C_2 W_1(\eta_1,\eta_2).\] Here $C_{Lf}$ is a Lipschitz constant for the function $f$, while $C_2=C_{Lf}C_1$. By the Gronwall's inequality, we conclude that 
	\begin{equation}\label{distr:ineq:x_eta}
		\|x_1(\cdot)-x_2(\cdot)\|\leq C_3 W_1(\eta_1,\eta_2).
	\end{equation} Here $C_3$ is a constant. Furthermore, 
	\[\begin{split}
		|w_1(t)-w_2(t)|&\leq \int_0^t |w_1(t')g(t',x_1(t'),\taverage{\eta_1}{t'})-w_2(t')g(t',x_2(t'),\taverage{\eta_1}{t'})| \\ &
		\leq \int_0^t\big[|w_1(t')-w_2(t')|\cdot|g(t',x_1(t'),\taverage{\eta_1}{t'})|\\&{}\hspace{50pt}+|w_2(t)|\cdot |g(t',x_1(t'),\taverage{\eta_1}{t'})-g(t',x_2(t'),\taverage{\eta_2}{t'})|\big]dt'.
	\end{split}\] Using the facts that $g$ is bounded by $C_g$ and Lipschitz continuous, the inequality $|w_2(t')|\leq C_1$ and estimate~\eqref{distr:ineq:x_eta}, we have that 
	\[|w_1(t)-w_2(t)|\leq C_g\int_0^t |w_1(t')-w_2(t')|dt'+C_4 W_1(\eta_1,\eta_2).\] This and the Gronwall's inequality give the following estimate for some constant $C_5$:
	\begin{equation}\label{distr:ineq:w_eta}
		|w_1(\cdot)-w_2(\cdot)|\leq C_5 W_1(\eta_1,\eta_2).
	\end{equation}
	Now, to estimate $W_1(\Phi(\eta_1),\Phi(\eta_2))$, we consider the  plan between $\Phi(\eta_1)$ and $\Phi(\eta_2)$ defined by the rule: 
	\[\tilde{\pi}\triangleq (\operatorname{traj}_{\eta_1},\operatorname{traj}_{\eta_2})\sharp\pi,\] where $\pi$ is an optimal plan between $\eta_1$ and $\eta_2$. From~\eqref{distr:ineq:x_eta} and~\eqref{distr:ineq:w_eta}, we conclude that 
	\[W_1(\Phi(\eta_1),\Phi(\eta_2))\leq C_6 W_1(\eta_1,\eta_2),\] where $C_6$ is a constant determined only on  $f$, $g$, $\|m_0\|$ and $T$. This gives the continuity of $\Phi$. Therefore, by the Schauder fixed point theorem (see \cite[17.56 Corollary]{Infinite_dimensional}), the mapping~$\Phi:\mathscr{P}\rightarrow\mathscr{P}$ admits a fixed point $\eta^*$. By construction of  $\Phi$, we have that $\eta^*$ is an equilibrium distribution of  curves with weights. Furthermore, since $\eta^*\in \mathscr{P}\subset \mathcal{P}^1(\widetilde{\Gamma}_T)$, the following estimate holds true: \[\|\taverage{\eta^*}{t}\|\leq C_1=\|m_0\|e^{C_gt}.\]
	
	Now, let us show that the equilibrium distribution of the weighted curves is unique. Let $\eta_1,\eta_2$ be two such distributions. This means that 
	\[\eta_1=\Phi(\eta_1),\ \ \eta_2=\Phi(\eta_2).\]
	For each $t$, set  
	\[\varrho(t)\triangleq  \mathcal{W}_{1,b}(\taverage{\eta_1}{t},\taverage{\eta_2}{t}).\] 
	
	Now let  $(x_1(\cdot,x_0),w_1(\cdot,x_0))=\operatorname{traj}_{\eta_1}(x_0)$, $(x_2(\cdot,x_0),w_2(\cdot,x_0))=\operatorname{traj}_{\eta_2}(x_0)$. 
	We have that \[\begin{split}\|x_1(t,&x_0)-x_2(t,x_0)\|\\&\leq \|x_1(0)+x_2(0)\|+ \int_{0}^t(C_{Lf}\|x_1(t',x_0)-x_2(t',x_0)\| +C_2\varrho(t'))dt'.\end{split}\] This and the Gronwall's inequality imply that 
	\begin{equation}\label{distr:ineq:x_t_varrho}
		\|x_1(t,x_0)-x_2(t,x_0)\|\leq C_7\int_0^t \varrho(t')dt'.
	\end{equation}
	Using the same arguments, one can show that 
	\begin{equation}\label{distr:ineq:w_t_varrho} |w_1(t,x_0)-w_2(t,x_0)|\leq C_8\int_0^t \varrho(t')dt'.\end{equation}
	
	Now, notice that
	\[\begin{split}
		\mathcal{W}_{1,b}(\taverage{\eta_1}{t},\taverage{\eta_2}{t})\leq \int_{\rd} (\|&x_1(t,x_0)-x_2(t,x_0)\|\\&+b| w_1(t,x_0)-w_2(t,x_0)|) \|m_0\|^{-1}m_0(dx_0).\end{split}\]  From this,~\eqref{distr:ineq:x_t_varrho},~\eqref{distr:ineq:w_t_varrho}, we conclude that 
	\[
	\varrho(t)=\mathcal{W}_{1,b}(\taverage{\eta_1}{t},\taverage{\eta_2}{t})\leq C_9\int_0^t \varrho(t')dt'.\] Using the Gronwall's inequality once more, we deduce that 
	\[\varrho(t)=\mathcal{W}_{1,b}(\taverage{\eta_1}{t},\taverage{\eta_2}{t})\equiv 0.\] Plugging this equality back into \eqref{distr:ineq:x_t_varrho} and \eqref{distr:ineq:w_t_varrho}, we deduce that  $\operatorname{traj}_{\eta_1}=\operatorname{traj}_{\eta_2}$.  Using the definition of the operator $\Phi$ (see \eqref{distr:intro:Phi}), we have that $\Phi(\eta_1)=\Phi(\eta_2)$. Therefore, we have that $\eta_1=\eta_2$.
	
\end{proof}

\section{Superposition principle for the nonlocal balance equation}\label{sect:superposition}

First, let us recall the definition of the weak solution to balance equation~\eqref{intro:eq:nonlocal_continuity}.

\begin{definition}\label{def:weak_solution}
	We say that a continuous flow of measures $m(\cdot):[0,T]\rightarrow \mathcal{M}(\rd)$ is a weak solution of balance equation~\eqref{intro:eq:nonlocal_continuity} if, for each $\phi\in C^1_c([0,T]\times\rd)$ and $s\in [0,T]$, the following equality holds: 
	\[\begin{split}
		\int_{\rd}\phi(s,&x)m(s,dx)-\int_{\rd}\phi(0,x)m_0(dx)\\&=\int_0^s\int_{\rd}\big[\partial_t\phi(t,x)+\langle \nabla\phi(x),f(t,x,m(t))\rangle+ \phi(x)g(t,x,m(t))\big]m(t,dx)dt.\end{split}\]
\end{definition}

\begin{remark} Using the same methods as in \cite[Lemma 8.1.2]{Ambrosio}, one can prove that $m(\cdot)$ is a weak solution of~\eqref{intro:eq:nonlocal_continuity} if and only, for every $\phi\in C^1_0((0,T)\times\rd)$,
	\[\int_0^T\int_{\rd}\big[\partial_t\phi(t,x)+\langle \nabla\phi(x),f(t,x,m(t))\rangle+ \phi(x)g(t,x,m(t))\big]m(t,dx)dt=0.\]
\end{remark}

The main result of this section is the following superposition principle. There, 	 $C>C_1$, where $C_1$ is introduced in Theorem~\ref{th:existence}.

\begin{theorem}\label{th:superposition}
	If $\eta\in \Gamma_T^C$ is an equilibrium distribution of  curves with time-varying weights, then $m(t)\triangleq \taverage{\eta}{t}$ is a weak solution of nonlocal balance equation~\eqref{intro:eq:nonlocal_continuity} on $[0,T]$. Conversely, if $m(\cdot)$ is a weak solution of~\eqref{intro:eq:nonlocal_continuity} satisfying the initial condition $m(0)=m_0$, then there exists an equilibrium distribution of   curves with time-varying weights $\eta$ such that $\taverage{\eta}{t}=m(t)$. In particular, there exists a unique solution of the initial value problem for  balance equation~\eqref{intro:eq:nonlocal_continuity}.
\end{theorem} 

\begin{remark} As we mentioned in the Introduction, papers \cite{Piccoli2014,Piccoli2019,Pogodaev2022} contain the existence and uniqueness results for the  balance equation with the right-hand side given by an arbitrary signed measure. However, they, in particular, assume that the source/sink term is uniformly continuous w.r.t. to the measure variable. This assumption is not directly applicable for the examined balance equation~\eqref{intro:eq:nonlocal_continuity}. 
\end{remark}

The proof of Theorem \ref{th:superposition} relies on the following uniqueness result. 
\begin{lemma}\label{lm:unique} Let $v:[0,T]\times\rd\rightarrow\rd$, $z:[0,T]\times\rd\rightarrow\mathbb{R}$ be continuous w.r.t. time variable. Additionally, assume that $v$ is Lipschitz continuous w.r.t. $x$. Then, the initial value problem for the linear balance equation:
	\begin{equation}\label{super:eq:PDE}
		\begin{split}
			&\partial_t m(t)+\operatorname{div}(v(t,x)m(t))=z(t,x)m(t),\\ &m(0)=m_0
		\end{split}
	\end{equation} allows at most one weak solution.
\end{lemma}
This lemma was proved in \cite[Lemma 3.5]{Maniglia2007}. However, for the sake of completeness,  we give its proof in Appendix \ref{appendix:lm_unique}.

\begin{proof}[Proof of Theorem~\ref{th:superposition}] 
	First, let us prove the necessity part.
	
	Let $\eta$ be an equilibrium distribution of weighted curves. By Proposition~\ref{prop:evaluation_lipschitz}, the mapping  $t\mapsto\taverage{\eta}{t}$ is continuous. Now, let  $\phi\in C^1_c([0,T]\times \rd)$. The definition of the operation $\taverage{\eta}{t}$ gives that, for every $s\in [0,T]$,
	\[
	\begin{split}
		\int_{\rd}\phi&(s,x)\taverage{\eta}{s}(dx)-	\int_{\rd}\phi(0,x)\taverage{\eta}{0}(dx)\\&=\int_{\Gamma_T^C} \big[\phi(s,x(s))w(s)-\phi(0,x(0))w(0)\big]\eta(d(x(\cdot),w(\cdot)))\\
		&=\int_{\Gamma_T^C} \int_0^s\big[
		(\partial_t\phi(t,x(t))+\langle\phi(t,x(t)),\dot{x}(t)\rangle)w(t)\\&{}\hspace{140pt}+\phi(t,x(t))\dot{w}(t)\big]dt\eta(d(x(\cdot),w(\cdot)))\\&=
		\int_{\Gamma_T^C} \int_0^s\Big[
		(\partial_t\phi(t,x(t))+\langle\phi(t,x(t)),f(t,x(t),\taverage{\eta}{t})\rangle)w(t)\\&{}\hspace{120pt}+\phi(t,x(t))g(t,x(t),\taverage{\eta}{t})w(t)\Big]dt\eta(d(x(\cdot),w(\cdot)))\\ &=
		\int_0^s\int_{\rd} \Big[
		(\partial_t\phi(t,x)+\langle\phi(t,x),f(t,x,\taverage{\eta}{t})\rangle)+\phi(t,x)g(t,x,\taverage{\eta}{t})\Big]\taverage{\eta}{t}(dx)dt.
	\end{split}
	\] 
	
	Therefore, the flow of measures $t\mapsto\taverage{\eta}{t}$ is a weak solution of nonlocal balance equation~\eqref{intro:eq:nonlocal_continuity}.
	
	To prove the sufficiency part, given $m(\cdot)$ that is  a weak solution of balance equation~\eqref{intro:eq:nonlocal_continuity}, we construct a distribution of weighted  curves $\eta$ corresponding to the linear balance equation~\eqref{super:eq:PDE} with 
	\begin{equation*}\label{super:intro:v_z}v(t,x)\triangleq f(t,x,m(t)),\ \ z(t,x)\triangleq g(t,x,m(t)).\end{equation*} By the necessity part of the theorem, the mapping $t\mapsto\taverage{\eta}{t}$ solves~\eqref{super:eq:PDE}. Lemma~\ref{lm:unique} yields that $\taverage{\eta}{t}\equiv m(t)$. This and construction of $\eta$ give that $\eta$ is a distribution of weighted curves corresponding to~\eqref{intro:eq:nonlocal_continuity} and the initial distribution $m_0$.
\end{proof}

\section{Conservation form of the balance equation}\label{sect:particle_balance}
In this section, we rewrite balance equation~\eqref{intro:eq:nonlocal_continuity} as a differential equation on a flow of probability measures defined on $\extRd$. The right-hand side of this equation is given by a L\'evy-Khintchine generator that combines a drift part and jumps between points on $\rd$ and the remote point~$\star$ introduced in \S~\ref{subsect:general_Wasserstein}.  
Natural interpretation of this process is a mean field particle system consisting of infinitely many identical particles with evolution of each particle combining drift and jumps. 

Notice that, due to Theorem~\ref{th:existence} and equation~\eqref{finite:eq:ODE}, we can assume that $m(\cdot)$ is such that $\|m(t)\|\leq R$, where $R> ce^{C_gT}$.
Denoting $\mathcal{M}^1_R(\rd)\triangleq \{m\in\mathcal{M}^1(\rd):\|m\|\leq R\}$, we have that the mapping $m\mapsto \mu\triangleq R^{-1}\mext{m}{R}$ provides the isomorphism between $\mathcal{M}^1_R(\rd)$ and $\mathcal{P}^1(\extRd)=\{\mu\in\mathcal{P}(\extRd):\int_{\rd}\|x\|\mu(dx)<\infty\}$. Notice that the inverse mapping is $\mu \mapsto R\mu|_{\rd}$, where $\mu|_{\rd}$ denotes the restriction of the probability $\mu$ on $\rd$, i.e., for each $\Upsilon\in \mathcal{B}(\rd)$,
\[(\mu|_{\rd})(\Upsilon)\triangleq \mu(\Upsilon).\] Recall that by~\eqref{intro:intro:m_ext}  and Proposition~\ref{prop:W_p_b}, for each two measures $m_1,m_2\in \mathcal{M}^1_R(\rd)$ and $\mu_1\triangleq R^{-1}\mext{m_1}{R}$, $\mu_2\triangleq R^{-1}\mext{m_2}{R}$, 
\[\mathcal{W}_{1,b}(m_1,m_2)=RW_1(\mu_1,\mu_2).\]

Now, for $t\in [0,T]$, $x\in\extRd$, $\mu\in\mathcal{P}(\extK)$, put
\begin{equation*}\label{markov:intro:f_R}
	f_R(t,x,\mu)\triangleq \left\{\begin{array}{ll}
		f(t,x,R\mu|_{\rd}), & x\in\rd,\\
		0, & x=\star,
	\end{array}\right.
\end{equation*}
\begin{equation*}\label{markov:intro:g_R}
	g_R(t,x,\mu)\triangleq \left\{\begin{array}{ll}
		g(t,x,R\mu|_{\rd}), & x\in\rd,\\
		0, & x=\star.
	\end{array}\right.
\end{equation*}

In the Introduction, we mentioned that the function $g$ combines probability rates for the particle to disappear or to give a spring. Thus, it is convenient to distinguish the positive and negative parts of $g$:
\[g^+(t,x,m)\triangleq g(t,x,m)\vee 0,\ \ g^-(t,x,m)\triangleq (-g(t,x,m))\vee 0.\]
Similarly,
\[g^+_R(t,x,\mu)\triangleq g_R(t,x,\mu)\vee 0,\ \ g^-_R(t,x,\mu)\triangleq (-g_R(t,x,\mu))\vee 0.\] Given   $\mu\in\mathcal{P}(\extRd)$, $m\triangleq R\cdot (\mathbbm{1}_{\rd}\sharp\mu)$, the quantity \[Rg^+_R(t,y,\mu)(\mathbbm{1}_{\rd}\sharp\mu)(dy)\Delta t=g^+(t,y,m)m(dy)\Delta t\] approximately describes  the distribution of newly appeared particles   on the time interval $[t,t+\Delta t]$, while $g^-_R(t,x,\mu)\Delta t+o(\Delta t)$ is the probability of the particle that occupies the state $x$ at time $t$ to disappear on the time interval $[t,t+\Delta t]$.  

For each $t\in [0,T]$, $x\in\extRd$, $\mu\in \mathcal{P}(\extRd)$, let $\nu^-(t,x,\mu,\cdot)$ be a finite measure on $\{0,1\}$ such that \begin{equation}\label{particle:intro:nu_minus}
	\nu^-(t,x,\mu,\{1\})=
	g^-_R(t,x,\mu),\ \ \nu^-(t,x,\mu,\{0\})=0.\end{equation}

Further, denote by $\nu^+(t,x,\mu,\cdot)$ the measure on $\extRd$ defined by the rule: for $\phi\in C_b(\extRd)$,
\begin{equation}\label{particle:intro:nu_plus_star}	
	\int_{\extRd}\phi(y)\nu^+(t,m,\star,dy)\triangleq \mu^{-1}(\{\star\})\int_{\rd}\phi(y)g^+_R(t,y,\mu)\mu(dy)\end{equation}  and 
\begin{equation}\label{particle:intro:nu_plus_td}
	\int_{\extRd}\phi(y)\nu^+(t,x,\mu,dy)=0
\end{equation} when $x\in\rd$.

With some abuse of notation, we denote by $C^1(\extRd)$ the set of functions $\phi$ from $\extRd$ to $\mathbb{R}$ such that the restriction of $\phi$ on $\rd$ lies in $C^1(\rd)$. 

Let us consider the following generator on $C^1(\extRd)$:
\begin{equation}\label{particle:intro:L}
	\begin{split}
		L_t[\mu]\phi(x)\triangleq \langle f_R&(t,x,\mu),\nabla\phi(x)\rangle\\&+\int_{\{0,1\}}(\phi(x+u(\star-x))-\phi(x)) \nu^-(t,x,\mu,du)\\
		&+\int_{\extRd} (\phi(y)-\phi(x))\nu^+(t,x,\mu,dy).	\end{split}
\end{equation} 

Now let us consider the following equation 	
\begin{equation}\label{particle:eq:generator}
	\frac{d}{dt}\mu(t)=L_t^*[\mu(t)]\mu(t),
\end{equation} where $L_t^*[\mu]$ is the operator adjoint to $L_t[\mu]$.

The solution of \eqref{particle:eq:generator} is considered in the space of probability measures. In particular, it preserves the total mass equal to 1.   

\begin{definition}\label{def:generator_L} We say that $\mu(\cdot):[0,T]\rightarrow \mathcal{P}^1(\extRd)$ solve~\eqref{particle:eq:generator} provided that, for each $\phi\in C^1_0(\extRd)$, 
	\[\int_{\extRd}\phi(x)\mu(t,dx)-\int_{\extRd}\phi(x)\mu(0,dx)=\int_0^t\int_{\extRd}L_\tau[\mu(\tau)]\phi(x)\mu(\tau,dx)d\tau.\]
\end{definition}

\begin{theorem}\label{th:restriction_kynetic} 
	A flow of probabilities $\mu(\cdot)$ is a solution of the equation~\eqref{particle:eq:generator}
	if and only if $m(\cdot)$ defined by the rule $m(t)\triangleq R\mu(t)|_{\rd}$ satisfies balance equation~\eqref{intro:eq:nonlocal_continuity}. In particular, for each $\mu_0\in \mathcal{P}^1(\extRd)$, there exists a unique solution to equation \eqref{particle:eq:generator} with the initial condition $\mu(0)=\mu_0\triangleq R^{-1}\mext{m_0}{R}$.
\end{theorem}
\begin{proof} 
	
	Choose $\mu\in\mathcal{P}(\extRd)$ and set $m\triangleq R\mu|_{\rd}$. 	
	Notice that each function $\phi\in C^1(\extRd)$ can be expressed in the form 
	\[\phi(x)=\tilde{\phi}(x)\mathbbm{1}_{\rd}(x)+c\mathbbm{1}_{\{\star\}}(x),\] where $\tilde{\phi}$ is $C^1$ function defined on $\rd$, while $c$ is a constant.
	Therefore, due to the definitions of functions $f_R$, $g_R$ and measures $\nu^-$, $\nu^+$, we obtain that 
	\[\begin{split}
		L_t[\mu]\phi(x)=\langle f(t,x,m),\tilde{\phi}(x)\rangle+(c&-\tilde{\phi}(x)\mathbbm{1}_{\rd}(x))g^-(t,x,m) \\+R^{-1}\mu^{-1}(\{\star\})&\int_{\rd}(\tilde{\phi}(y)-c)g^+(t,y,m)m(dy)\mathbbm{1}_{\{\star\}}(x).\end{split}\]
	Thus, due to the fact that $g(t,x,m)=g^+(t,x,m)-g^-(t,x,m)$, we have the equality
	\begin{equation}\label{particle:equality:R_gen}
		\begin{split}
			R\int_{\extRd}L_t[\mu]&\phi(x)\mu(dx)\\=\int_{\rd}\langle &f(t,x,m),\tilde{\phi}(x)\rangle m(dx)\\&+\int_{\rd}\tilde{\phi}(x)g(t,x,m)m(dx) +c\int_{\rd}g(t,x,m)m(dx).\end{split}
	\end{equation}
	Now let $\mu(\cdot)$ solve  equation~\eqref{particle:eq:generator}.  We define $m(t)\triangleq R(\mu(t)|_{\rd})$. 
	Using this and~\eqref{particle:equality:R_gen} with $c=0$, we conclude that $m(\cdot)$ satisfies~\eqref{intro:eq:nonlocal_continuity}. 
	
	To prove the converse statement, we let $\mu_0\triangleq R^{-1}\mext{m(0)}{R}$ and find $\mu(\cdot)$ solving~\eqref{particle:eq:generator} with the initial condition $\mu(0)=\mu_0$. Therefore, $t\mapsto R(\mu(t)|_{\rd})$ is a solution of balance equation~\eqref{intro:eq:nonlocal_continuity}. The uniqueness result for this equation gives that $m(t)=R(\mu(t)|_{\rd})$.
	
\end{proof}

We complete this section with the mean field particle system interpretation of the balance equation. 

\begin{remark}\label{def:representation}
	It is natural to say that a 5-tuple $(\Omega,\mathcal{F},\{\mathcal{F}_t\}_{t\in [0,T]},\mathbb{P},X)$ is a mean field particle system representation of balance equation~\eqref{intro:eq:nonlocal_continuity} provided that
	\begin{itemize}
		\item $(\Omega,\mathcal{F},\{\mathcal{F}_t\}_{t\in [0,T]},\mathbb{P})$ is a probability space with filtration;
		\item $X$ is a $\{\mathcal{F}_t\}_{t\in [0,T]}$-adapted stochastic process with values in $\extRd$;
		\item for every $\phi\in C^1(\extRd)$, the process
		\[\phi(X(t))-\int_0^t L_{\tau}[\mu(\tau)]\phi(X(\tau))d\tau\] is a  $\{\mathcal{F}_t\}_{t\in [0,T]}$-martingale   such that 	
		\begin{equation}\label{particle:equality:m}
			\mu(t)=X(t)\sharp\mathbb{P}.
		\end{equation}
	\end{itemize}
	Indeed, one can regard the stochastic process $X$ as a dynamics of a sampling particle, while the probability $\mu(t)$ gives a distribution of all particles in time $t$. 
	
	By construction, we have that  $\mu(\cdot)$ defined by~\eqref{particle:equality:m} satisfies \eqref{particle:eq:generator}, while $m(\cdot)$ defined by the rule $m(t)=R\mu(t)|_{\rd}$ solves~\eqref{intro:eq:nonlocal_continuity}.
	
	One can show that there exists at least one particle system representation of \eqref{intro:eq:nonlocal_continuity}. The accurate proof of this fact goes beyond the scope of the paper. The scheme of the proof follows the way introduced in \cite{Kolokoltsov_Markov}.  We assume that we are given with the flow of probabilities $\mu(\cdot)$ that solves ~\eqref{particle:equality:m}.  First, we discretize the dynamics and approximate the drift part of the dynamics $f_R(t,x,\mu)$ by a Kolmogorov matrix. Then, we obtain a continuous time Markov chain. The corresponding stochastic process is uniformly stochastically continuous \cite[Theorem 5.4.1]{Kolokoltsov_Markov}. Finally, one can argue as in \cite[Theorem 5.4.1]{Kolokoltsov_Markov} and show that, when the discretization parameter tends to zero, the continuous time Markov chains converge. The limit provides the desired particle system representation.
\end{remark}

\section{Approximating system of ODEs}\label{sect:finite}
In this section, we introduce  a system of ODEs that arises when we approximate the drift term $f$ by a Kolmogorov matrix.

Let $\mathcal{S}$ be  a finite  subset of $\rd$. The fineness of $\mathcal{S}$ is evaluated by a number \begin{equation*}\label{finite:intro:d_S}
	d(\mathcal{S})\triangleq\min\big\{\|{x}-{y}\|:\, {x},{y}\in\mathcal{S},\, {x}\neq{y}\big\}.
\end{equation*} the diameter of $\mathcal{S}$ is
\[\operatorname{diam}(\mathcal{S})=\max\big\{\|{x}-{y}\|:\, {x},{y}\in\mathcal{S},\, {x}\neq{y}\big\}.\]

A signed measure on $\mathcal{S}$ is always determined by a sequence $\beta_{\mathcal{S}}=(\beta_{{x}})_{{x}\in\mathcal{S}})\subset \mathbb{R}$. We denote the set of such sequences by $\mathscr{l}_1(\mathcal{S})$. It is endowed with the norm 
\[\|\beta_{\mathcal{S}}\|\triangleq \sum_{{x}\in\mathcal{S}}|\beta_{{x}}|.\] 
We will assume that  each such sequence is a row-vector. Further, the set of elements of $\mathscr{l}_1(\mathcal{S})$ with nonnegative entries is denoted by~$\mathscr{l}_1^+(\mathcal{S})$. This space inherits the metric from~$\mathscr{l}_1(\mathcal{S})$.

The mapping
\begin{equation*}\label{finite_stae:intro:isomorph}
	\beta_{\mathcal{S}}\mapsto \mathscr{I}(\beta_{\mathcal{S}})\triangleq \sum_{{x}\in\mathcal{S}} \beta_{{x}}\delta_{{x}}
\end{equation*} provides the isomorphism between $\mathscr{l}_1^+(\mathcal{S})$ and
$\mathcal{M}(\mathcal{S})$.

Recall that, on $\mathcal{M}(\mathcal{S})$, we consider the metric $\mathcal{W}_{1,b}$. 

The distance between images w.r.t. the isomorphism $\mathscr{I}$ is evaluated in the following statement which is proved in the Appendix~\ref{appendix:distance}.

\begin{proposition}\label{prop:metric_equivalence} Let $b\geq \operatorname{diam}(\mathcal{S})$. Then, the following estimates hold true:
	\begin{equation}\label{finite:ineq:W_metric}b^{-1} \mathcal{W}_{1,b}(\mathscr{I}(\beta^1),\mathscr{I}(\beta^2))\leq \|\beta^1-\beta^2\|\leq (d(\mathcal{S}))^{-1} \mathcal{W}_{1,b}(\mathscr{I}(\beta^1),\mathscr{I}(\beta^2)).\end{equation}
\end{proposition}

Now let us introduce a system of ODEs indexed by elements of  $\mathscr{l}_1^+(\mathcal{S})$. In Section~\ref{sect:coupled} we 	will show that it  approximates  original nonlocal balance equation~\eqref{intro:eq:nonlocal_continuity}. 
To this end, for ${x}\in\mathcal{S}$, set
\[  \hat{g}_{{x}}(t,\beta_{\mathcal{S}})\triangleq g(t,{x},\mathscr{I}(\beta_{\mathcal{S}})).\] 
Further, let $Q(t,\beta_{\mathcal{S}})=(Q_{{x},{y}}(t,\beta_{\mathcal{S}}))_{{x},{y}\in\mathcal{S}}$ be a Kolmogorov matrix, i.e.,
for each ${x}\in\mathcal{S}$, $Q_{{x},{y}}(t,\beta_{\mathcal{S}})\geq 0$ when ${x}\neq{y}$, whereas
\[\sum_{{y}\in\mathcal{S}}Q_{{x},{y}}(t,\beta_{\mathcal{S}})=0.\]

The approximating system takes the following form:
\[\frac{d}{dt}\beta_{{y}}(t)=\sum_{{x}\in\mathcal{S}}\beta_{{x}}(t) Q_{{x},{y}}(t,\beta_{\mathcal{S}}(t)) + \beta_{{y}}(t)\hat{g}_{{y}}(t,\beta_{\mathcal{S}}(t)).\] Since we assume that $\beta_{\mathcal{S}}=(\beta_{{x}})_{{x}\in\mathcal{S}}$ is a row-vector, one can rewrite this system in the vector form:
\begin{equation}\label{finite:eq:ODE}
	\frac{d}{dt}\beta_{\mathcal{S}}(t)=\beta_{\mathcal{S}}(t)Q(t,\beta_{\mathcal{S}}(t))+\beta_{\mathcal{S}}(t)G(t,\beta_{\mathcal{S}}(t)).
\end{equation} Here, we denote by $G(t,\beta_{\mathcal{S}})=(G_{{x},{y}}(t,\beta_{\mathcal{S}}))_{{x},{y}\in\mathcal{S}}$ the diagonal matrix determined by the rule:
\begin{equation*}\label{finite:intro:G}
	G_{{x},{y}}(t,\beta_{\mathcal{S}})=\left\{
	\begin{array}{lc}
		\hat{g}_{{x}}(t,\beta_{\mathcal{S}}), & {x}={y},\\
		0,& {x}\neq {y}.
	\end{array}
	\right.
\end{equation*}

Due to the classical existence and uniqueness theorem for ODEs, system~\eqref{finite:eq:ODE} has a unique solution for every initial condition.

To rewrite equation~\eqref{finite:eq:ODE} in the conservation form, we first extend the phase space to $\extS$ and, for each $t\in [0,T]$, $\beta_{\mathcal{S}}\in \mathscr{l}_1^+(\mathcal{S})$, define the  matrix $\mathcal{Q}(t,\beta_{\mathcal{S}})=(\mathcal{Q}_{{x},{y}}(t,\beta_{\mathcal{S}}))_{{x},{y}\in\mathcal{S}\cup\{\star\}}$ by the rule: 
\begin{equation}\label{Markov_ode:intro:Q_extened}
	\mathcal{Q}_{{x},{y}}(t,\beta_{\mathcal{S}})\triangleq \left\{\begin{array}{ll}
		Q_{{x},{y}}(t,\beta_{\mathcal{S}}), & {x},{y}\in\mathcal{S},\, {x}\neq{y}\\
		\hat{g}^-_{{x}}(t,\beta_{\mathcal{S}}), & {x}\in \mathcal{S},\, {y}=\star,\\
		(R-\|\beta\|)^{-1}\hat{g}^+_{{y}}(t,\beta_{\mathcal{S}})\beta_{{y}}, & {x}=\star,\, {y}\in\mathcal{S},\\
		Q_{{x},{x}}(t,\beta_{\mathcal{S}})-\hat{g}^-_{{x}}(t,\beta_{\mathcal{S}}), & {x}={y}\in\mathcal{S},\\
		-(R-\|\beta\|)^{-1}\sum_{{z}\in\mathcal{S}}g^+_{{z}}(t,\beta_{\mathcal{S}})\beta_{{z}}, & {x}={y}=\star.
	\end{array}\right.\end{equation} One can directly check that $\mathcal{Q}(t,\beta_{\mathcal{S}})$ is a Kolmogorov matrix.

Further, we denote the set of sequences of nonnegative numbers indexed by elements of $\mathcal{S}\cup\{\star\}$ by $\mathscr{l}_1^+(\extS)$. As above, $\mathscr{I}$ is the isomorphism between $\mathscr{l}_1^+(\extS)$ and $\mathcal{M}(\mathcal{S}\cup\{\star\})$ given by the rule: for $\beta_{\mathcal{S},\star}=(\beta_{{x}})_{{x}\in\mathcal{S}\cup\{\star\}}$, 
\[\mathscr{I}(\beta_{\mathcal{S},\star})\triangleq \sum_{{x}\in\mathcal{S}\cup\{\star\}}\beta_{{x}}\delta_{{x}}.\] 
If $\beta_{\mathcal{S},\star}=(\beta_{{x}})_{{x}\in\mathcal{S}\cup\{\star\}}\in \mathscr{l}_1^+(\extS)$, $\beta_{\mathcal{S}}$ is its restriction on $\mathcal{S}$, then 
\[ \mathcal{Q}(t,\beta_{\mathcal{S},\star})\triangleq \mathcal{Q}(t,\beta_{\mathcal{S}}).\]

Due to the definition of the matrix $\mathcal{Q}$ (see~\eqref{Markov_ode:intro:Q_extened}), we have that, if $\beta_{\mathcal{S},\star}(\cdot)=(\beta_{{x}}(\cdot))_{{x}\in\mathcal{S}\cup\{\star\}}$ satisfies
\begin{equation*}\label{finite:eq:ODE_extended}
	\frac{d}{dt}\beta_{\mathcal{S},\star}(t)=\beta_{\mathcal{S},\star}(t)\mathcal{Q}(t,\beta_{\mathcal{S},\star}(t)),
\end{equation*} then $\beta_{\mathcal{S}}(\cdot)\triangleq (\beta_{{x}}(\cdot))_{{x}\in\mathcal{S}}$ solves~\eqref{finite:eq:ODE}.

Now let us introduce the description of system~\eqref{finite:eq:ODE} through the generator technique. As above, we will use the probability measures. Notice that each $\beta_{\mathcal{S},\star}$ corresponds to the probability measure $\mu\triangleq R^{-1}\mathscr{I}(\beta_{\mathcal{S},\star})$.

If $t\in [0,T]$, ${x}\in\extS$, $\mu\in \mathcal{P}(\extS)$, then 
\begin{equation}\label{Markov_ode:intro:L_S}
	\begin{split}
		L^{Q}_t[\mu]\phi({x})\triangleq \sum_{{y}\in\mathcal{S}}[&\phi({y}) -\phi({x})]Q_{{x},{y}}(t,\mathscr{I}^{-1}(R\mu|_{\mathcal{S}}))\mathbbm{1}_{\mathcal{S}}({x})\\&+[\phi(\star)-\phi({x})]\hat{g}_{{x}}^-(t,\mathscr{I}^{-1}(R\mu|_{\mathcal{S}}))\mathbbm{1}_{\mathcal{S}}({x}) \\&+\mu^{-1}(\{\star\})\sum_{{y}\in\mathcal{S}}[\phi(y)-\phi(\star)] \hat{g}_{{y}}^+(t,\mathscr{I}^{-1}(R\mu|_{\mathcal{S}}))\mu(\{{y}\})\mathbbm{1}_{\{\star\}}({x}).	
	\end{split}
\end{equation} Using the definitions of the measures of the $\nu^-$ and $\nu^+$ (see~\eqref{particle:intro:nu_minus}--\eqref{particle:intro:nu_plus_td}), one can rewrite this definition as 
\begin{equation*}\label{Markov_ode:intro:L_S_eq}
	\begin{split}
		L^{Q}_t[\mu]\phi({x})\triangleq \sum_{{y}\in\mathcal{S}}[\phi({y})-\phi({x})]&Q_{{x},{y}}(t,\mathscr{I}^{-1}(R\mu|_{\mathcal{S}}))\mathbbm{1}_{\mathcal{S}}\\+\int_{\extS} [\phi(x&+u(\star-x))-\phi({x})]\nu^-(t,{x},\mu,du) \\&+\int_{\extS}[\phi(y)-\phi(x)] \nu^+(t,{x},\mu,dy).
	\end{split}
\end{equation*}

The conservation form of  ~\eqref{finite:eq:ODE} is the equation 
\begin{equation}\label{Markov_ode:eq:generator}
	\frac{d}{dt}\mu(t)=L_t^{Q,*}[\mu(t)]\mu(t).
\end{equation} Here $L_t^{Q,*}[\mu]$ stands for the operator adjoint to $L_t^Q[\mu]$. Its solution is considered in the weak sense. 

\begin{proposition}\label{prop:generator_finite} If $\mu_0\in\mathcal{P}(\extS)$, then there exists a unique solution to the initial value problem for equation~\eqref{Markov_ode:eq:generator} 
	and the initial condition $\mu(0)=\mu_0$. Furthermore,  $\mu(\cdot)$ solves~\eqref{Markov_ode:eq:generator} if and only if the function $\beta_{\mathcal{S}}(\cdot)$ defined by the rule $\beta_{\mathcal{S}}(t)\triangleq R\mathscr{I}^{-1}(\mu(t)|_{\mathcal{S}})$ satisfies ~\eqref{finite:eq:ODE}.
\end{proposition}
The proposition directly follows from the definition of the generator $L^Q$ (see~\eqref{Markov_ode:intro:L_S}) and the matrix $\mathcal{Q}$.

As in the case of balance equation, system of ODEs~\eqref{finite:eq:ODE} admits a particle representation.
\begin{remark}\label{remark:particle_Markov}
	A 5-tuple $(\Omega,\mathcal{F},\{\mathcal{F}\}_{t\in [0,T]},\mathbb{P},X^Q)$ is said to be a particle representation of~\eqref{finite:eq:ODE} if
	\begin{itemize}
		\item $(\Omega,\mathcal{F},\{\mathcal{F}_t\}_{t\in [0,T]},\mathbb{P})$ is a probability space with filtration;
		\item $X^Q$ is a stochastic process with values in $\mathcal{S}\cup \{\star\}$;
		\item for each $\phi\in C(\mathcal{S}\cup\{\star\})$,
		\[\phi(X^Q(t))-\int_0^t L^Q_{\tau}[\mu(\tau)]\phi(X^Q(\tau))d\tau\] is a  $\{\mathcal{F}_t\}_{t\in [0,T]}$-martingale with $\mu(\cdot)$ defined by the rule:  for each $t\in [0,T]$, $A\in \mathcal{B}(\mathcal{K})$, 
		\begin{equation}\label{Markov_ODE:equality:m}
			\mu(t)=(X^Q(t))\sharp \mathbb{P}.
		\end{equation} 
	\end{itemize} 
	
	As above, $\mu(\cdot)$ defined by~\eqref{Markov_ODE:equality:m} satisfies~\eqref{Markov_ode:eq:generator}. Thus, $\beta_{\mathcal{S}}(\cdot)$ such that $\beta_{\mathcal{S}}(t)=R\mathscr{I}^{-1}(\mu(t)|_{\mathcal{S}})$ is a solution of~\eqref{finite:eq:ODE}.
	
	The existence of the particle representation of~\eqref{finite:eq:ODE} directly follows from \cite[Theorem 5.4.2]{Kolokoltsov_Markov}
	
\end{remark}

\section{Rate of approximation of balance equation}\label{sect:coupled}

\subsection{Formulation of the approximation theorem}
For simplicity, we assume that the balance equation operates on some compact. This means that we impose the following condition.
\begin{enumerate}[label=(A\arabic*), resume*=main_cond]
	\item\label{cond:main:compact} There exists a compact $\mathcal{K}$ such that $\operatorname{supp}(m_0)\subset \mathcal{K}$ and, for each $t\in [0,T]$, $x\in \rd\setminus\mathcal{K}$, $m\in\mathcal{M}(\mathcal{K})$, 
	\[f(t,x,m)=0.\]
\end{enumerate}  
The following property directly follow from the imposed conditions.
\begin{proposition}\label{prop:K} Assume conditions~\ref{cond:main:firts_moment}--\ref{cond:main:compact}. If $m(\cdot)$ is a solution to balance equation~\eqref{intro:eq:nonlocal_continuity} with the initial condition $m(0)=m_0$, then, for every $t\in [0,T]$,
	\[\operatorname{supp}(m(t))\subset\mathcal{K}.\]
\end{proposition}

It suffices to define the functions $f_R$, $g^+$, $g^-$, $g_R+$, $g_R^-$ and the measure $\nu^-$ only for $x\in\mathcal{K}$. Furthermore,  $\nu^+(t,\star,\mu,\cdot)$ is now a measure on $\extK$ defined by the rule: for $\phi\in C_b(\extK)$,
\begin{equation}\label{particle:intro:nu_plus_star_K}	
	\int_{\extK}\phi(y)\nu^+(t,m,\star,dy)\triangleq \mu^{-1}(\{\star\})\int_{\mathcal{K}}\phi(y)g^+_R(t,y,\mu)\mu(dy).\end{equation}   
Finally, $L$ takes the form
\begin{equation*}
	\begin{split}
		L_t[\mu]\phi(x)\triangleq \langle f_R&(t,x,\mu),\nabla\phi(x)\rangle\\&+\int_{\{0,1\}}(\phi(x+u(\star-x))-\phi(x)) \nu^-(t,x,\mu,du)\\
		&+\int_{\extK} (\phi(y)-\phi(x))\nu^+(t,x,\mu,dy).	\end{split}
\end{equation*}

The approximation result  relies the following  conditions on $\mathcal{S}$ and $Q$: there exists $\varepsilon>0$ such that 
\begin{enumerate}[label=(QS\arabic*)]
	\item\label{cond:torus} for each $x\in\mathcal{K}$,
	\[\min_{{y}\in\mathcal{S}}\|x-{y}\|\leq\varepsilon;\]
	\item\label{cond:f_Q} for every $t\in [0,T]$, ${x}\in\mathcal{S}$, $\beta_{\mathcal{S}}\in \mathscr{l}_1^+(\mathcal{S})$,
	\[\Bigg\|f(t,{x},\mathscr{I}(\beta_{\mathcal{S}}))-\sum_{{y}\in\mathcal{S}}({y}-{x})Q_{{x},{y}}(t,\beta_{\mathcal{S}})\Bigg\|\leq \varepsilon;\] 
	\item\label{cond:variance} for every $t\in [0,T]$, ${x}\in\mathcal{S}$, $\beta_{\mathcal{S}}\in \mathscr{l}_1^+(\mathcal{S})$,
	\[\sum_{{y}\in\mathcal{S}}\|x-y\|^2 Q_{{x},{y}}(t,\beta_{\mathcal{S}})\leq\varepsilon^2.\] 
\end{enumerate}

It is reasonable to assume that $\varepsilon$ is sufficiently small. Without loss of generality, we will consider the case where $\varepsilon\leq 1$. Additionally, we assume that $b>\operatorname{diam}(\mathcal{K})+1\geq \operatorname{diam}(\mathcal{K}\cup\mathcal{S})$.

\begin{theorem}\label{th:approx} Assume that conditions~\ref{cond:torus}--\ref{cond:variance} are in force. Given $c>0$, there exists a constant $\widehat{C}>0$ determined by functions $f$, $g$ and the constant $c$ such that, if  
	\begin{itemize}
		\item $\beta_0\in\mathscr{l}_1^+(\mathcal{S})$, $\|\beta_0\|_1\leq c$;
		\item $m_0\in\mathcal{M}(\mathcal{K})$, $\|m_0\|\leq c$;
		\item $\beta_{\mathcal{S}}(\cdot):[0,T]\rightarrow \mathscr{l}_1^+(\mathcal{S})$ solves~\eqref{finite:eq:ODE} with the initial condition $\beta_{\mathcal{S}}(0)=\beta_0$;
		\item $m(\cdot):[0,T]\rightarrow \mathcal{M}(\mathcal{K})$ satisfies~\eqref{intro:eq:nonlocal_continuity} and the initial condition $m(0)=m_0$,
	\end{itemize}  then
	\[\mathcal{W}_{1,b}(m(t),\mathscr{I}(\beta_{\mathcal{S}}(t)))\leq \widehat{C}(\varepsilon+\mathcal{W}_{1,b}(m_0,\mathscr{I}(\beta_0)).\] 
\end{theorem}
This theorem is proved in \S~\ref{subsect:distance}. The proof relies on an auxiliary construction of the generator on $(\extK)\times\extS$ introduced in~$\S$\ref{subsect:generator}.

Before, we turn to this auxiliary construction, let us present  an example of the system that approximates balance equation~\eqref{intro:eq:nonlocal_continuity} with arbitrary small accuracy.

Let $h>0$, $K_h\triangleq K+B_h$,  
\begin{equation}\label{finite:intro:example_S}
	\mathcal{S}^h\triangleq \mathcal{K}_h\cap h\mathbb{Z}^d.
\end{equation} 
Here, $B_h$ is a closed ball centered in the origin of the radius $h$, $\mathbb{Z}$ states for the set of integers. Further, consider the coordinate-wise representation of the function $f$
\[f(t,x,m)=(f_1(t,x,m),\ldots,f_d(t,x,m)).\] Put
\begin{equation}\label{finite:intro:example_Q} Q^h_{{x},{y}}(t,\beta_{\mathcal{S}})\triangleq\left\{
	\begin{array}{ll}
		h^{-1}|f_i(t,{x},\mathscr{I}(\beta_{\mathcal{S}}m))|, & \begin{array}{l} {x}\in \mathcal{K},\\  \ \  {y}={x}+ he_i\operatorname{sgn}(f_i(t,x,\mathscr{I}(\beta_{\mathcal{S}}))),\end{array}\\
		-h^{-1}\sum_{i=1}^{d}|f_i(t,{x},\mathscr{I}(\beta_{\mathcal{S}}))|, & {x}={y}\in\mathcal{K},\\
		0, & \text{otherwise}. 
	\end{array}
	\right.\end{equation} Hereinafter, $e_i$ stands for the $i$-th coordinate vector, while $\operatorname{sgn}$ denotes the sign function defined by the rule:
\[\operatorname{sgn}(a)\triangleq \left\{\begin{array}{ll}
	1, & a>0,\\
	-1, & a<0,\\
	0, & a=0.
\end{array}\right.\]

\begin{proposition}\label{prop:example}
	The lattice $\mathcal{S}^h$ and the Kolmogorov matrix defined by~\eqref{finite:intro:example_S} and~\eqref{finite:intro:example_Q} respectively satisfy conditions~\ref{cond:torus}--\ref{cond:variance} with $\varepsilon=\max\{h,d^{1/4}C_f^{1/2}\sqrt{h}\}$.
\end{proposition}
\begin{proof}
	The proof mimes one given in~\cite{Averboukh2016}. Condition~\ref{cond:torus} directly follows from~\eqref{finite:intro:example_S}. 
	
	We have that 
	\[\sum_{{y}\in\mathcal{S}}({y}-{x})Q^h(t,x,\mathscr{I}(\beta_{\mathcal{S}}))=\sum_{i=1}^d f_i(t,x,\mathscr{I}(\beta_{\mathcal{S}}))e_i=f(t,x,\mathscr{I}(\beta_{\mathcal{S}})).\] This gives~\ref{cond:f_Q}.
	
	Finally, condition~\ref{cond:variance} follows from the estimate:
	\[\sum_{{y}\in\mathcal{S}}\|{y}-{x}\|^2Q^h(t,x,\mathscr{I}(\beta_{\mathcal{S}}))=h^{-1}\sum_{i=1}^dh^2|f_i(t,x,\mathscr{I}(\beta_{\mathcal{S}}))|\leq  C_f\sqrt{d}h. \] 
\end{proof}

\subsection{Coupled dynamics}\label{subsect:generator}
We choose a constant $R$ such that 
\[R>ce^{C_gT}.\]
Notice that, due to Theorem~\ref{th:existence} and equation~\eqref{finite:eq:ODE}, we can assume that $m(\cdot)$ and $\beta_\mathcal{S}(\cdot)$ defined in Theorem~\ref{th:approx} are such that $\|m(t)\|< R$, $\|\beta_{\mathcal{S}}(t)\|< R$.  

The key idea of the proof of Theorem \ref{th:approx} is to construct a generator $\Lambda^Q$ such that, if $\vartheta(\cdot):[0,T]\rightarrow \mathcal{P}^1((\extK)\times(\extS))$ solves 
\begin{equation}\label{coupled:eq:Lambda}
	\frac{d}{dt}\vartheta(t)=\Lambda^{Q,*}_t[\vartheta(t)]\vartheta(t),\end{equation} then the flows of probabilities $\mu_1(\cdot)$ and $\mu_2(\cdot)$ defined by the rule: \[\mu_i(t)\triangleq \operatorname{p}^i\sharp\vartheta(t), \ \ i=1,2\]  are solutions of~\eqref{particle:eq:generator} and~\eqref{Markov_ode:eq:generator} respectively. Hereinafter, $\Lambda^{Q,*}_t[\vartheta]$ is an operator adjoint to $\Lambda^Q_t[\vartheta]$. Notice that $\vartheta(t)$ is a plan between $\mu_1(t)$ and $\mu_2(t)$. We will evaluate $\int_{(\extK)\times (\extS)}\|x_1-x_2\|\vartheta(t,d(x_1,x_2))$. This together with the link between distance on $\mathcal{P}^1(\extK)$ and  metric $\mathcal{W}_{1,b}$ will give the proof of Theorem~\ref{th:approx}.

The generator $\Lambda^Q$ will describe the  behavior of the couple of particles where the first particle can move on $\mathcal{K}$, while the second one walks randomly on $\mathcal{S}$. Additionally, both particles can jump to and from the remote point $\star$. The probability rates of these synchronized jumps will be given by L\'{e}vy measures $\chi^-$ and $\chi^+$.  They are defined as follows.

Let $t\in [0,T]$, $x_1\in\extK$, $x_2\in\extS$ and let $\vartheta\in \mathcal{P}((\extK)\times(\extS))$. Denote $\mu_1\triangleq \operatorname{p}^1\sharp\vartheta$, $\mu_2\triangleq \operatorname{p}^2\sharp\vartheta$.

First, we choose the  measure   $\chi^-(t,x_1,x_2,\vartheta,\cdot)$ to be a measure on $\{0,1\}\times \{0,1\}$ satisfying
\begin{itemize}
	\item $\chi^-(t,x_1,x_2,\vartheta,\{(1,1)\})\triangleq g^-_R(t,x_1,\mu_1)\wedge g^-_R(t,x_2,\mu_2)$;
	\item 
	$\chi^-(t,x_1,x_2,\vartheta,\{(1,0)\})\triangleq g^-_R(t,x_1,\mu_1)-(g^-_R(t,x_1,\mu_1)\wedge g^-_R(t,x_2,\mu_2))$;
	\item 
	$\chi^-(t,x_1,x_2,\vartheta,\{(0,1)\})\triangleq g^-_R(t,x_2,\mu_2)-(g^-_R(t,x_1,\mu_1)\wedge g^-_R(t,x_2,\mu_2))$;
	\item $\chi^-(t,x_1,x_2,\vartheta,\{(0,0)\})\triangleq  0$.
\end{itemize} Recall that by definition $g^-(t,\star,\mu)\equiv 0$.

Now, let us introduce the measure  $\chi^+(t,x_1,x_2,\vartheta)$.   For every $\phi\in C_b((\extK)\times (\extS))$, we set
\begin{equation}\label{coupled:intro:chi_plus_st_st}
	\begin{split}
		\int_{(\extK)\times (\extS)}\phi(&y_1,y_2)\chi^+(t,\star,\star,\vartheta,d(y_1,y_2))\triangleq \\ (\vartheta\{(\star,\star)\})^{-1}&\int_{\mathcal{K}\times \mathcal{S}} \big[\phi(y_1,y_2)(g^+_R(t,y_1,\mu_1)\wedge g^+_R(t,y_2,\mu_2))\\+\phi(y_1,\star&)(g^+_R(t,y_1,\mu_1)-(g^+_R(t,y_1,\mu_1)\wedge g^+_R(t,y_2,\mu_2)))\\+\phi(\star,y_2&)(g^+_R(t,y_2,\mu_2)-(g^+_R(t,y_1,\mu_1)\wedge g^+_R(t,y_2,\mu_2)))\big]\vartheta(d(y_1,y_2)).
	\end{split}
\end{equation}

When $(x_1,x_2)\neq (\star,\star)$, we put 
\begin{equation}\label{coupled:intro:chi_plus_not_st_st}
	\chi^+(t,x_1,x_2,\vartheta,\cdot)\equiv 0.
\end{equation}

Notice that the synchronization of two jump-type processes was proposed in~\cite{Kolokoltsov_Markov} (see also~\cite{Khlopin2022}, where the synchronization of stopping time was considered).

Finally, let us define the generator $\Lambda^Q$ by the  rule: for $\phi\in C((\extK)\times \extS)$ such that $\phi$ is continuously differentiable w.r.t. the first variable on $\mathcal{K}$:
\begin{equation}\label{coupled:intro:generator}
	\begin{split}
		\Lambda^Q_t[\vartheta]\phi(x_1&,x_2)\triangleq \\\langle \nabla_{x_1}&\phi(x_1,x_2),f_R(t,x_1,\operatorname{p}^1\sharp\vartheta)\rangle \mathbbm{1}_{\mathcal{K}}(x_1) \\&+\sum_{{y}\in\mathcal{S}}\phi(x_1,{y})Q_{x_2,{y}}(t,\mathscr{I}^{-1}(R(\operatorname{p}^1\sharp\vartheta)|_{\mathcal{S}}))\mathbbm{1}_{\mathcal{S}}(x_2)\\&+
		\int_{\{0,1\}\times \{0,1\}}[\phi(x_1+u_1(\star-x_1),x_2+u_2(\star-x_2))-\phi(x_1,x_2)]\\&{\hspace{190pt}}\chi^-(t,x_1,x_2,\vartheta,d(u_1,u_2))\\&+
		\int_{(\extK)\times (\extS)}[\phi(y_1,y_2)-\phi(x_1,x_2)]\chi^+(t,x_1,x_2,\vartheta,d(y_1,y_2)).
	\end{split}
\end{equation} 


In the following, we, with some abuse of notation, denote by $C^1((\extK)\times(\extS))$ the set of functions $\phi:(\extK)\times(\extS)\rightarrow\mathbb{R}$ such that the mapping $x_1\mapsto \phi(x_1,x_2)$ is continuously differentiable on $\mathcal{K}$.

\begin{definition}\label{def:solution_Lambda} We say that a flow of probabilities $\vartheta:[0,T]\rightarrow\mathcal{P}^1((\extK)\times(\extS))$ is a solution to~\eqref{coupled:eq:Lambda} if, for each $\phi\in C^1((\extK)\times (\extS))$, one has that 
	\begin{equation}\label{coupled:equality:def_solution}
		\begin{split}
			\int_{(\extK)\times (\extS)}\phi(x_1,x_2)\vartheta(t,&d(x_1,x_2))\\&-\int_{(\extK)\times (\extS)}\phi(x_1,x_2)\vartheta(0,d(x_1,x_2))\\=\int_0^t \int_{(\extK)\times (\extS)}&(\Lambda^Q_\tau[\vartheta(\tau)]\phi(x_1,x_2))\vartheta(\tau,d(x_1,x_2))d\tau.\end{split}
	\end{equation}
\end{definition}

\begin{remark} As above, one can consider a particle representation of the flow $\vartheta(\cdot)$. It is a  6-tuple $(\Omega,\mathcal{F},\{\mathcal{F}_t\}_{t\in [0,T]},\mathbb{P},X,X^Q)$ such that the following conditions hold true: 
	\begin{itemize}
		\item $(\Omega,\mathcal{F},\{\mathcal{F}_t\}_{t\in [0,T]},\mathbb{P})$ is a probability space with filtration;
		\item $X$ and $X^Q$ are $\{\mathcal{F}_t\}_{t\in [0,T]}$-adapted processes taking values in $\extK$ and $\extS$ respectively;
		\item if $\vartheta(\cdot)$ is defined by the rule $\vartheta(t)\triangleq (X(t),X^Q(t))\sharp\mathbb{P}$, then the process
		\[\phi(X(t),X^Q(t))-\int_0^t\Lambda^Q_\tau[\vartheta(\tau)]\phi(X(\tau),X^Q(\tau))d\tau\] is a $\{\mathcal{F}_t\}_{t\in [0,T]}$-martingale.
	\end{itemize}  
	Notice that $X(\cdot)$ describes the continuous motion, while $X^Q$ is a Markov chain approximation. The first and second terms of the generator $\Lambda^Q$ provide the independent evolution on $\mathcal{K}$ and $\mathcal{S}$ respectively. Simultaneously, the measure $\chi^-$ gives the coordinated jumps  from $\mathcal{K}$ and $\mathcal{S}$ to the remote point $\star$. The reverse jumps are described by the measure $\chi^+$. It implies the jumps only if both components $X(t-)$ and $X^Q(t-)$ are in the remote point $\star$ and distribute the state $X(t),X^Q(t)$ according to the current measure $\vartheta(t)$. The way how one can construct a particle representation of a flow $\vartheta(\cdot)$ is briefly discussed in Remark~\ref{remark:particle_Markov}.
\end{remark}

It was mentioned above that the proof of Theorem~\ref{th:approx} relies on the properties of the solution to~\eqref{coupled:eq:Lambda}. Thus, we need an existence result.

\begin{proposition}\label{prop:theta_exists} Given an initial measure $\vartheta_0\in \mathcal{P}^1((\extK)\times(\extS))$, there exists a solution to~\eqref{coupled:eq:Lambda} that starts from $\vartheta_0$. If $\vartheta(\cdot)$ is a solution of~\eqref{coupled:eq:Lambda}, then, at any time $t$, $\vartheta(t)$ is concentrated on $(\extK)\times(\extS)$.
\end{proposition}  
\begin{proof}
	Notice that such existence results are analogous to one proved in~\cite{Kolokoltsov_Markov, Kolokoltsov_DE}. There, the approximation of the dynamics by Markov chains was used. However, for the sake of completeness, we give the direct proof. The approximation technique follows the way introduced in Section~\ref{sect:finite}.
	
	Let $a>0$, $K^a\triangleq K+B_a$, $\mathcal{S}^a\triangleq K_a\cap a\mathbb{Z}^d$. Further, if $x,y\in\mathcal{S}^a$, $\vartheta^a\in\mathcal{P}^1((\extSr)\times(\extS))$, put 
	\begin{equation}\label{coupled:intro:Q_a}
		\mathcal{Q}^a_{x,y}(t,\vartheta^a)\triangleq\left\{
		\begin{array}{ll}
			a^{-1}|f_i(t,x,\operatorname{p}^1\sharp\vartheta^a)|, & \begin{array}{l} x\in \mathcal{K},\\  \ \  y=x+ ae_i\operatorname{sgn}(f_i(t,x,\operatorname{p}^1\sharp\vartheta^a)),\end{array}\\
			-a^{-1}\sum_{i=1}^{d}|f_i(t,x,\operatorname{p}^1\sharp\vartheta^a)|, & x=y\in\mathcal{K},\\
			0, & \text{otherwise}. 
		\end{array}
		\right.	
	\end{equation}
	For each $t\in [0,T]$, $\vartheta^a\in\mathcal{P}^1((\extSr)\times(\extS))$, let us consider the following  generator:
	\[\begin{split}	\Lambda^a_t[\vartheta^a]\phi(x_1&,x_2)\triangleq \\\sum_{y\in\mathcal{S}^a}&\phi(y,x_2)\mathcal{Q}^a_{x_1,y}(t,\vartheta^a) +\sum_{{y}\in\mathcal{S}}\phi(x_1,{y})Q_{x_2,{y}}(t,\mathscr{I}^{-1}(R(\operatorname{p}^1\sharp\vartheta^a)|_{\mathcal{S}}))\\&+
		\int_{\{0,1\}\times \{0,1\}}[\phi(x_1+u_1(\star-x_1),x_2+u_2(\star-x_2))-\phi(x_1,x_2)]\\&{\hspace{210pt}}\chi^-(t,x_1,x_2,\vartheta^a,d(u_1,u_2))\\&+
		\int_{(\extK)\times (\extS)}[\phi(y_1,y_2)-\phi(x_1,x_2)]\chi^+(t,x_1,x_2,\vartheta^a,d(y_1,y_2)).
	\end{split}\]
	
	Due to~\cite[Theorem 7.2.1]{Kolokoltsov_DE}, for each $\vartheta_0^a$ that is concentrated on $(\extSr)\times(\extS)$, there exists a solution of the equation 
	\begin{equation}\label{coupled:eq:S_r_Lambda}
		\frac{d}{dt}\vartheta^a(t)=\Lambda^{a,*}_t[\vartheta^a(t)]\vartheta^a(t).
	\end{equation} As above, $\Lambda^{a,*}_t[\theta]$ stands for the operator adjoint to $\Lambda^a_t[\theta]$. The solution of~\eqref{coupled:eq:S_r_Lambda} is considered in the weak sense. Notice that, for each $t$, $\vartheta^a(t)$ is concentrated on $(\extSr)\times(\extS)$. Moreover, letting $\phi\equiv 1$, we conclude that each  measure $\vartheta^a(t)$ is a probability.
	
	Further, let $\phi\in \operatorname{Lip}_1((\extSr)\times(\extS))$, we have that 
	\[\Bigg|\int_{(\extSr)\times (\extS)}(\Lambda^a_t\phi(x_1,x_2)) \vartheta^a(t)\Bigg| \leq C_f+C_{10}\|\phi\|.\] Here $C_{10}$ is a constant that can depend on $\mathcal{Q}$ but does not depend on $a$. Notice that, without loss of generality, one can assume that $\phi(\star,\star)=0$ and, thus, for some constant $C_{11}$, $\|\phi\|\leq C_{11}$.  Therefore, for each $s,r\in [0,T]$, and $\phi\in\operatorname{Lip}_1((\extK)\times(\extS))$ such that $\phi(\star,\star)=0$,
	\[\begin{split}
		\Bigg|\int_{(\extK)\times(\extS)}\phi(&x_1,x_2)\vartheta^a(r,d(x_1,x_2))\\-&\int_{(\extK)\times(\extS)}\phi(x_1,x_2)\vartheta^a(s,d(x_1,x_2))\Bigg|\leq C_{12}(r-s).\end{split}\] As above, $C_{12}$ stands for a constant. Due to this estimate and the Kantorovich-Rubinstein duality~\cite{Ambrosio}, we have that the set of curves~$\{\vartheta^a(\cdot)\}$ are relatively compact in $C([0,T],\mathcal{P}^1((\mathcal{K}^1\cup\{\star\})\times (\extS))$. Assuming that $\{\vartheta^{a}_0\}$ converges to $\vartheta_0$, one can construct a sequence $\{a_n\}_{n=1}^\infty$ and a limiting curve $\vartheta(\cdot)$ such that $\vartheta(0)=\vartheta_0$, $a_n\rightarrow 0$ as $n\rightarrow\infty$, while $\{\vartheta^{a_n}(\cdot)\}$ converges to $\vartheta(\cdot)$ in $C([0,T],\mathcal{P}^1((\mathcal{K}^1\cup\{\star\})\times (\extS))$. 
	
	Let us prove that $\vartheta(\cdot)$ is a solution to~\eqref{coupled:eq:Lambda}. To this end, we prove  that, if 
	\begin{itemize}
		\item $\{\vartheta_n\}_{n=1}^\infty\subset\mathcal{P}^1((\extSr)\times (\extS))$,
		\item  $\vartheta\in\mathcal{P}^1((\extK)\times(\extS))$, 
		\item $\{a_n\}_{n=1}^\infty\subset (0,+\infty)$,
		\item $a_n\rightarrow 0$  $\vartheta_n\rightarrow\vartheta$ as $n\rightarrow\infty$,
	\end{itemize} then, for each $x_1\in\mathcal{K}^1\cup\{\star\}$, $x_2\in\extS$, $t\in [0,T]$, $\phi\in C^1((\mathcal{K}^1\cup\{\star\})\times(\extS))$,
	\begin{equation}\label{coupled:convergence:Lambda}
		\Lambda^{a_n}_t[\vartheta_n]\phi(x_1,x_2)\rightarrow \Lambda^Q_t[\vartheta]\phi(x_1,x_2).
	\end{equation}
	
	First, notice that the dependence
	\begin{equation*}\label{coupled:convergence:lasts}
		\begin{split}
			\vartheta\mapsto
			\sum_{{y}\in\mathcal{S}}\phi(&x_1,{y})Q_{x_2,{y}}(t,\mathscr{I}^{-1}(R(\operatorname{p}^1\sharp\vartheta)|_{\mathcal{S}}))\\&+
			\int_{\{0,1\}\times \{0,1\}}[\phi(x_1+u_1(\star-x_1),x_2+u_2(\star-x_2))-\phi(x_1,x_2)]\\&{\hspace{190pt}}\chi^-(t,x_1,x_2,\vartheta,d(u_1,u_2))\\&+
			\int_{(\extK)\times (\extS)}[\phi(y_1,y_2)-\phi(x_1,x_2)]\chi^+(t,x_1,x_2,\vartheta,d(y_1,y_2))
		\end{split}
	\end{equation*} 
	is continuous for each $t\in [0,T]$, $x_1\in \mathcal{K}$, $x_2\in\mathcal{S}$, $\phi\in C((\extK)\times(\extS))$. 
	
	Thanks to the choice of the function $\phi$, the modulus of continuity $\varsigma(\cdot)$ of the dependence  $\mathcal{K}\ni x_1\mapsto\nabla_{x_1}\phi(x_1,x_2)$ is well-defined. In particular, $\varsigma(a)\rightarrow 0$ as $a\rightarrow 0$. Further,
	\[|a^{-1}(\phi(x_1\pm ae_i,x_2)-\phi(x_1,x_2))-\langle\nabla_{x_1}\phi(x_1,x_2),e_1\rangle|\leq \varsigma(a).\] Therefore, by~\eqref{coupled:intro:Q_a}, we have that, for $x_1\in\mathcal{K}^1$, $x_2\in\extS$,
	\[\Bigg|\sum_{y\in\mathcal{S}^a}\phi(y,x_2)\mathcal{Q}^a_{x_1,y}(\vartheta)-\big\langle f(t,x_1,\operatorname{p}^1\sharp\vartheta),\nabla_{x_1}\phi(x_1,x_2)\big\rangle\Bigg|\leq C_f\varsigma(a).\]
	Combining this with the  continuity of the last three terms in the formulae for the generators $\Lambda^Q$ (see~\eqref{coupled:intro:generator}) and the continuity of the function $f$, we obtain the convergence result~\eqref{coupled:convergence:Lambda}.
	Thus, we can pass to the limit in~\eqref{coupled:eq:S_r_Lambda} and conclude that $\vartheta(\cdot)$ solves~\eqref{coupled:eq:Lambda}.
	
	The fact that $\vartheta(t)$ is concentrated on $(\extK)\times(\extS)$  directly follows from the fact that $\Lambda^Q[\vartheta]$ is concentrated on $(\extK)\times(\extS)$ while $\vartheta$ is concentrated on this set. Finally, choosing $\phi^a(x_1,x_2)\equiv 1$, we conclude that $\vartheta(t,(\extK)\times (\extS))=1$. This means that each measure $\vartheta(t)$ is a probability.
	
\end{proof}

Further, we show that the first and  second marginal distributions of the solution of~\eqref{coupled:eq:Lambda} solve~\eqref{particle:eq:generator},~\eqref{Markov_ode:eq:generator} respectively. 

\begin{proposition}\label{prop:generator_coupling} Let $\vartheta(\cdot)$ solve the equation~\eqref{coupled:eq:Lambda}
	and let $\vartheta(0)$ be concentrated on $(\extK)\times(\extS)$. Then,
	\begin{itemize}
		\item the flow of probabilities $\mu_1(\cdot)$  defined by the rule $\mu_1(t)\triangleq \operatorname{p}^1\sharp(\vartheta(t))$ satisfies equation~\eqref{particle:eq:generator};
		\item the flow of probabilities $\mu_2(\cdot)$ such that $\mu_2(t)\triangleq \operatorname{p}^2\sharp(\vartheta(t))$ is a solution of~\eqref{Markov_ode:eq:generator}.
	\end{itemize}
\end{proposition}
\begin{proof}
	To prove the first statement, we assume that $\phi$ depends only on $x_1\in\extK$. For each $\vartheta\in\mathcal{P}^1((\extK)\times(\extS))$, put $\mu_1\triangleq \operatorname{p}^1\sharp\vartheta$. We have that 
	\begin{equation}\label{coupled:Lambda:x_1:1}
		\begin{split}
			\Lambda^Q_t[\vartheta]\phi&(x_1)\triangleq \langle \nabla\phi(x_1),f_R(t,x_1,\mu_1)\rangle\\&+
			\int_{\{0,1\}\times \{0,1\}}[\phi(x_1+u_1(\star-x_1))-\phi(x_1)]\chi^-(t,x_1,x_2,\vartheta,d(u_1,u_2))\\& +
			\int_{(\extK)\times (\extS)}[\phi(y_1)-\phi(x_1)]\chi^+(t,x_1,x_2,\vartheta,d(y_1,y_2)).	
		\end{split}
	\end{equation}

	Further, recall that 
	\begin{equation}\label{coupled:Lambda:x_1:2}
		\begin{split}\int_{\{0,1\}\times \{0,1\}}[\phi(x_1+u_1&(\star-x_1))-\phi(x_1)]\chi^-(t,x_1,x_2,\vartheta,d(u_1,u_2))\\&= \int_{\{0,1\}}[\phi(x_1+u_1(\star-x_1))-\phi(x_1)]\nu^-(t,x_1,\mu_1,du_1).\end{split}
	\end{equation} 
	
	Additionally, 
	from the definition of the measure $\chi^+$ (see~\eqref{coupled:intro:chi_plus_st_st},~\eqref{coupled:intro:chi_plus_not_st_st}), we have that 
	\begin{equation*}
		\begin{split}
			\int_{(\extK)\times (\extS)}[\phi(y_1)-\phi(x_1)&]\chi^+(t,x_1,x_2,\vartheta,d(y_1,y_2))\\=
			(\vartheta\{(\star,\star)\})^{-1}\int_{\mathcal{K}\times \mathcal{S}} (\phi(y_1&)-\phi(\star))g^+_R(t,y_1,\mu_1)\vartheta(d(y_1,y_2))\mathbbm{1}_{\{(\star,\star)\}}(x_1,x_2).
		\end{split}
	\end{equation*} Integrating the both parts of this equality against the measure $\vartheta(d(x_1,x_2))$, we have that
	\begin{equation*}\label{coupled:equality:integral_Lambda}
		\begin{split}
			\int_{(\extK)\times (\extS)}\int_{\mathcal{K}\times \mathcal{S}}[\phi(y_1)-\phi(x_1)]\chi^+(t,x_1,x_2,\vartheta,d&(y_1,y_2))\vartheta(d(x_1,x_2))\\=
			\int_{\extK} (\phi(y_1)-\phi&(\star))g^+_R(t,y_1,\mu_1)\mu_1(dy_1).
		\end{split}
	\end{equation*} Using this and the definition of the measure $\nu^+$ (see~\eqref{particle:intro:nu_plus_star},~\eqref{particle:intro:nu_plus_td},~\eqref{particle:intro:nu_plus_star_K}), we conclude that 
	\begin{equation*}
		\begin{split}
			\int_{(\extK)\times (\extS)}\int_{\mathcal{K}\times \mathcal{S}}[\phi(y_1)-\phi(x_1)]&\chi^+(t,x_1,x_2,\vartheta,d(y_1,y_2))\vartheta(d(x_1,x_2))\\=\int_{\extK}\int_{\extK}&[\phi(y_1)-\phi(x_1)]\nu^+(t,x_1,\mu_1,dy_1)d\mu_1(dx_1).
		\end{split}
	\end{equation*}
	
	This expression and formulae~\eqref{coupled:Lambda:x_1:1},~\eqref{coupled:Lambda:x_1:2} yield the equality
	\begin{equation}\label{couple:equality_theta_mu}
		\begin{split}
			\int_{(\extK)\times (\extS)}&\Lambda^Q_t[\vartheta]\phi(x_1)\vartheta(d(x_1,x_2))=\int_{\mathcal{K}}\langle\nabla\phi(x_1),f_R(t,x_1,\mu_1)\rangle\\&+\int_{\extK}\int_{\{0,1\}}[\phi(x_1+u_1(\star-x_1))-\phi(x_1)]\nu^-(t,x_1,\mu_1,du_1)\\&+\int_{\extK}\int_{\extK}[\phi(y_1)-\phi(x_1)]\nu^+(t,x_1,\mu_1,dy_1)\mu_1(dx_1)\\ =
			&\int_{\extK}L_t[\mu_1]\phi(x)\mu_1(dx).
		\end{split}
	\end{equation} Now, recall that $\vartheta(\cdot)$ satisfies~\eqref{coupled:eq:Lambda}, while $\mu_1(\cdot)$ is defined by the rule: $\mu_1(t)=\operatorname{p}^1\sharp\vartheta(t)$. This,~\eqref{couple:equality_theta_mu} and the definition of the generator $L_t$ (see~\eqref{particle:intro:L}) yield that, if $\phi\in C^1(\extK)$, then
	\[\frac{d}{dt}\int_{\extK}\phi(x)\mu_1(dx)=\int_{\extK}L_t[\mu_1]\phi(x)\mu_1(dx).\]
	This implies the first statement of the proposition. 
	
	The second statement is proved in the same way.
\end{proof}

\subsection{Regularized distance between marginal distribution}\label{subsect:distance}

In this section, we give the proof of Theorem \ref{th:approx}. As we mentioned above, the natural idea here is to study the evolution of the integral $\int_{(\extK)\times(\extS)}\|x_1-x_2\|\vartheta(t,d(x_1,x_2))$ using equation~\eqref{coupled:equality:def_solution}. However, the distance function is not smooth. Therefore, we regularize it and evaluate the action of the generator $\Lambda^Q_t[\vartheta]$ on the function 
\begin{equation*}\label{coupled:intro:varphi_alpha}
	\varphi_\varepsilon(x_1,x_2)\triangleq \sqrt{\|x_1-x_2\|^2+\varepsilon^2}.
\end{equation*}

In the following lemmas, we gradually estimate the action of the parts of the generator $\Lambda^Q$ on this function. These results will imply Theorem \ref{th:approx}. We start with the evaluation of the action of first two terms.

\begin{lemma}\label{lm:dist_f_Q} For each $t\in [0,T]$, $x_1\in \mathcal{K}$, ${x}_2\in \mathcal{S}$, $\vartheta\in \mathcal{P}((\extK)\times(\extS))$, and $\mu_1\triangleq \operatorname{p}^1\sharp\vartheta$, $\mu_2\triangleq \operatorname{p}^2\sharp\vartheta$,
	\[\begin{split}
		\Bigl|\langle f_R(t,x_1,\mu_1),\nabla_{x_1}\varphi_\varepsilon(x_1,{x}_2)\rangle + \sum_{{y}\in\mathcal{S}}&\varphi_\varepsilon(x_1,{y}) Q_{{x}_2,{y}}(t,\mathscr{I}^{-1}(R\mu_2|_{\mathcal{S}}))\Bigr|\\&\leq C_{Lf}(\varphi_\varepsilon(x_1,{x}_2)+RW_1(\mu_1,\mu_2))+2\varepsilon.\end{split}\]
	
\end{lemma}
\begin{proof}
	First, we have that 
	\begin{equation}\label{coupled:equality:derivative_phi}
		\nabla_{x_1}\varphi_\varepsilon(x_1,x_2)=\frac{x_1-x_2}{\sqrt{\|x_1-x_2\|^2+\varepsilon^2}}.
	\end{equation} Expanding the function $y\mapsto \varphi_\varepsilon(x_1,y)$ using the Taylor series   about the point $x_2$, we have that 
	\begin{equation}\label{coupled:equality_tailor}
		\begin{split}
			\varphi_\varepsilon(x_1,y)=\varphi_\varepsilon(x_1,x_2)&-
			\frac{\langle x_1-x_2,y-x_2\rangle}{\sqrt{\|x_1-x_2\|^2+\varepsilon^2}}\\&+ \frac{\|y-x_2\|^2(\|y'-x_1\|^2+\varepsilon^2)-(\langle y'-x_1,y-x_2\rangle)^2}{2(\|y'-x_1\|^2+\varepsilon^2)^{3/2}}.\end{split}
	\end{equation} Here $y'$ lies in the segment $[x_1,y]$. Since $|\langle x_2-x_1,y-x_2\rangle|\leq \|y-x_2\|\cdot\|x_2-x_1\|$, we conclude that, for each $y\neq x_2$,
	\[\begin{split}
		\Bigg|&\frac{\|y-x_2\|^2(\|y'-x_1\|^2+\varepsilon^2)-(\langle y'-x_1,y-x_2\rangle)^2}{2(\|y'-x_1\|^2+\varepsilon^2)^{3/2}}\Bigg|\\&{}\hspace{30pt}\leq 
		\frac{\|y-x_2\|^2(\|y'-x_1\|^2+\varepsilon^2)}{2(\|y'-x_1\|^2+\varepsilon^2)^{3/2}}+\frac{|\langle y'-x_1,y-x_2\rangle|^2}{2(\|y'-x_1\|^2+\varepsilon^2)^{3/2}}\leq \|y-x_2\|^2\varepsilon^{-1}.
	\end{split}\]
	Using this and~\eqref{coupled:equality_tailor}, we conclude that
	\[
	\Bigg|\varphi_\varepsilon(x_1,y)- \Bigg(\varphi_\varepsilon(x_1,x_2)-
	\frac{\langle x_1-x_2,y-x_2\rangle}{\sqrt{\|x_1-x_2\|^2+\varepsilon^2}}\Bigg)\Bigg|\leq \varepsilon^{-1}\|y-x_2\|^2.
	\]
	Since $Q(t,\mathscr{I}^{-1}(R\mu_2|_{\mathcal{S}}))$ is a Kolmogorov matrix, we derive that
	\begin{equation*}
		\begin{split}
			\Bigg|\sum_{{y}\in\mathcal{S}}\varphi_\varepsilon(x_1,{y})&Q_{x_2,{y}}(t,\mathscr{I}^{-1}(R\mu_2|_{\mathcal{S}}))\\&+\sum_{{y}\in\mathcal{S},{y}\neq x_2}
			\frac{\langle x_1-x_2,y-x_2\rangle}{\sqrt{\|x_1-x_2\|^2+\varepsilon^2}}Q_{x_2,{y}}(t,\mathscr{I}^{-1}(R\mu_2|_{\mathcal{S}}))\Bigg|\\
			\leq\varepsilon^{-1}\sum_{{y}\in\mathcal{S}}&\|y-x_2\|^2Q_{x_2,{y}}(t,\mathscr{I}^{-1}(R\mu_2|_{\mathcal{S}})).
		\end{split}
	\end{equation*}	Due to~\ref{cond:variance}, the right-hand side of this inequality is bounded by $\varepsilon$. Thus,
	\begin{equation}\label{coupled:ineq:phi_2}
		\begin{split}
			\Bigg|\sum_{{y}\in\mathcal{S}} \varphi_\varepsilon(x_1,{y})Q_{x_2,{y}} &(t,\mathscr{I}^{-1}(R\mu_2|_{\mathcal{S}}))\\&+\sum_{{y}\in\mathcal{S},{y}\neq x_2}
			\frac{\langle x-x_2,y-x_2\rangle}{\sqrt{\|x-x_2\|^2+ \varepsilon^2}}Q_{x_2,{y}}(t,\mathscr{I}^{-1}(R\mu_2|_{\mathcal{S}}))\Bigg| \leq\varepsilon.\end{split}
	\end{equation}
	The Lipschitz continuity of the function $f$ gives that
	\begin{equation}
		\begin{split}
			\Bigl|\langle f_R(t,x_1,\mu_2),\nabla_{x_1}\varphi_\varepsilon(x_1,x_2)\rangle + \sum_{{y}\in\mathcal{S}}\varphi_\varepsilon(x_1,&{y}) Q_{x_2,{y}}(t,\mathscr{I}^{-1}(R\mu_2|_{\mathcal{S}}))\Bigr|\\ \leq 
			\Bigl|\langle f_R(t,x_2,\mu_2),\nabla_{x_1}\varphi_\varepsilon(x_1,x_2)\rangle + \sum_{{y}\in\mathcal{S}}&\varphi_\varepsilon(x_1,{y}) Q_{x_2,{y}}(t,\mathscr{I}^{-1}(R\mu_2|_{\mathcal{S}}))\Bigr|\\&+C_{Lf}(\|x_1-x_2\|+RW_{1}(\mu_1,\mu_2)).
		\end{split}
	\end{equation} Taking into account~\eqref{coupled:equality:derivative_phi} and~\eqref{coupled:ineq:phi_2}, we conclude that 
	\begin{equation*}\label{coupled:ineq:f_Q_1}
		\begin{split}
			\Bigl|\langle f_R(t,x_1,\mu_1),\nabla_{x_1}\varphi_\varepsilon(x_1,&x_2)\rangle + \sum_{{y}\in\mathcal{S}}\varphi_\varepsilon(x_1,{y}) Q_{x_2,{y}}(t,\mathscr{I}^{-1}(R\mu_2|_{\mathcal{S}}))\Bigr|\\ \leq \frac{1}{\sqrt{\|x_1-x_2\|^2+\varepsilon^2}}
			&{}\Bigg|\Bigl\langle f_R(t,x_2,\mu_2)\\&{}\hspace{20pt}-\sum_{{y}\in\mathcal{S}}(y-x_2)Q_{x_2,{y}}(t,\mathscr{I}^{-1}(R\mu_2|_{\mathcal{S}})),x_1-x_2\Bigr\rangle \Bigg|\\&+
			C_{Lf}(\|x_1-x_2\|+RW_{1}(\mu_1,\mu_2))+\varepsilon.
		\end{split}
	\end{equation*} Condition~\ref{cond:f_Q} and the definition of the function $f_R$ yield that 
	\[\Bigg|f_R(t,x_2,\mu_2)-\sum_{{y}\in\mathcal{S}}(y-x_2)Q_{x_2,{y}}(t,\mathscr{I}^{-1}(R\mu_2|_{\mathcal{S}}))\Bigg|\leq \varepsilon.\] Therefore,
	\begin{equation*}
		\begin{split}
			\Bigg|\langle f_R(t,x_1,&\mu_1),\nabla_{x_1}\varphi_\varepsilon(x_1,x_2)\rangle + \sum_{{y}\in\mathcal{S}}\varphi_\varepsilon(x_1,{y}) Q_{x_2,{y}}(t,\mathscr{I}^{-1}(R\mu_2|_{\mathcal{S}}))\Bigg|\\ &\leq \varepsilon\frac{\|x_1-x_2\|}{\sqrt{\|x_1-x_2\|^2+\varepsilon^2}}+
			C_{Lf}(\|x_1-x_2\|+RW_{1}(\mu_1,\mu_2))+\varepsilon.
		\end{split}
	\end{equation*} This, together with the inequality $\|x_1-x_2\|\leq \sqrt{\|x_1-x_2\|^2+\varepsilon^2}=\varphi_\varepsilon(x_1,x_2)$ imply the statement of the lemma.
\end{proof}

The following lemma gives an evaluation of the third term of the generator in the case when the test function is equal to $\varphi_\varepsilon$.
\begin{lemma}\label{lm:action_chi_minus}
	There exists a constant $C_{13}$ depending only on $f$ and $g$ such that, for each $t\in [0,T]$, $x_1,x_2\in\extK$, $\vartheta\in \mathcal{P}^1((\extK)\times (\extS))$, $\mu_1\triangleq\operatorname{p}^1\sharp\vartheta$, $\mu_2\triangleq \operatorname{p}^2\sharp\vartheta$, one has
	\[\begin{split}\Bigg|\int_{\{0,1\}\times \{0,1\}} [\varphi_\varepsilon(x_1+u_1(\star-x_1),x_2
		&+u_2(\star-x_2)) \\-\varphi_\varepsilon(x_1,x_2)&]\chi^-(t,x_1,x_2,\vartheta,d(u_1,u_2))\Bigg|\\ \leq \varepsilon C_g+C_{13}\varphi_\varepsilon(x_1,x_2)+C_{13}&RW_1(\mu_1,\mu_2).\end{split}\]
\end{lemma}
\begin{proof} 
	Direct computations yield that 
	\[\begin{split}\int_{\{0,1\}\times \{0,1\}} [\varphi_\varepsilon&(x_1+u_1(\star-x_1),x_2
		+u_2(\star-x_2)) -\varphi_\varepsilon(x_1,x_2)]\\&{}\hspace{171pt}\chi^-(t,x_1,x_2,\vartheta,d(u_1,u_2))
		\\=(\varphi_\varepsilon(\star&,\star)-\varphi_\varepsilon(x_1,x_2))(g_R^-(t,x_1,\mu_1)\wedge g_R^-(t,x_2,\mu_2))\\+&(\varphi_\varepsilon(\star,x_2)-\varphi_\varepsilon(x_1,x_2))(g_R^-(t,x_1,\mu_1)-(g_R^-(t,x_1,\mu_1)\wedge g_R^-(t,x_2,\mu_2)))
		\\+&(\varphi_\varepsilon(x_1,\star)-\varphi_\varepsilon(x_1,x_2))(g_R^-(t,x_2,\mu_2)-(g_R^-(t,x_1,\mu_1)\wedge g_R^-(t,x_2,\mu_2))).
	\end{split} 
	\] Further, taking into account the facts that $\varphi_\varepsilon(\star,\star)=\varepsilon$, while $\varphi_\varepsilon(x_1,\star)=\varphi_\varepsilon(\star,x_2)=\sqrt{b^2+\varepsilon^2}$, we have that 
	\[\begin{split}\int_{\{0,1\}\times \{0,1\}} &[\varphi_\varepsilon(x_1+u_1(\star-x_1),x_2
		+u_2(\star-x_2)) -\varphi_\varepsilon(x_1,x_2)]\\&{}\hspace{171pt}\chi^-(t,x_1,x_2,\vartheta,d(u_1,u_2))
		\\=\varepsilon(g_R^-&(t,x_1,\mu_1)\wedge g_R^-(t,x_2,\mu_2))\\+&\sqrt{b^2+\varepsilon^2}((g_R^-(t,x_1,\mu_1)\vee g_R^-(t,x_2,\mu_2))-(g_R^-(t,x_1,\mu_1)\wedge g_R^-(t,x_2,\mu_2)))
		\\-&\varphi_\varepsilon(x_1,x_2)(g_R^-(t,x_1,\mu_1)\vee g_R^-(t,x_2,\mu_2)).
	\end{split} 
	\] The deduced equality, the boundedness of the function $g^-_R$ as its Lipschitz continuity w.r.t. the phase and measure variables give that 
	\[\begin{split}\Bigg|\int_{\{0,1\}\times \{0,1\}} [\varphi_\varepsilon(x_1+u_1(\star-x_1),x_2
		+u_2(\star-x_2)) &-\varphi_\varepsilon(x_1,x_2)]\\&\chi^-(t,x_1,x_2,\vartheta,d(u_1,u_2))\Bigg|
		\\ \leq\varepsilon C_g+\sqrt{b^2+\varepsilon^2}C_{Lg}(\|x_1-x_2\|+RW_1(\mu_1&,\mu_2))
		+\varphi_\varepsilon(x_1,x_2)C_g.
	\end{split} 
	\] This implies the conclusion of the lemma.
\end{proof}

Now let us evaluate the action of the measure $\chi^+$ on the regularized distance function. 
\begin{lemma}\label{lm:chi_plus} There exists a  constant $C_{14}$ such that, for each $t\in[0,T]$, $\vartheta\in \mathcal{P}((\extK)\times (\extS))$, $\mu_1\triangleq\operatorname{p}^1\sharp\vartheta$, $\mu_2\triangleq\operatorname{p}^2\sharp\vartheta$, the following estimate holds true:
	\[
	\begin{split}
		\Bigg|\int_{(\extK)\times (\extS)}[&\varphi_\varepsilon(y_1,y_2) -\varphi_\varepsilon(\star,\star)]\chi^+(t,\star,\star,\vartheta,d(y_1,y_2))\Bigg|\\ \leq
		\vartheta^{-1}(\{\star,&\star\})\Bigg[C_g\varepsilon+C_{14}\int_{\mathcal{K}\times \mathcal{S}}\varphi_\varepsilon(y_1,y_2)\vartheta(d(y_1,y_2))+C_{14}RW_1(\mu_1,\mu_2)\Bigg].
	\end{split}
	\]
\end{lemma}
\begin{proof} 
	By the definition of the measure $\chi^+(t,\star,\star,\vartheta,\cdot)$ (see~\eqref{coupled:intro:chi_plus_st_st}), we have that 
	\[\begin{split}
		\vartheta(\{\star,\star\})&\int_{(\extK)\times (\extS)}\varphi_\varepsilon(y_1,y_2)\chi^+(t,\star,\star,\vartheta,d(y_1,y_2))\\=
		\int_{\mathcal{K}\times \mathcal{S}}&\big[\varphi_\varepsilon(y_1,y_2)(g_R^+(t,y_1,\mu_1)\wedge g_R^+(t,y_2,\mu_2))\\&+\sqrt{b^2+\varepsilon^2}((g_R^+(t,y_1,\mu_1)\vee g_R^+(t,y_2,\mu_2))-(g_R^+(t,y_1,\mu_1)\wedge g_R^+(t,y_2,\mu_2)))\big]\\&\hspace{278pt}\vartheta(d(y_1,y_2)).	\end{split}
	\]
	The boundedness and the Lipschitz continuity of the function $g$ (and, thus, the function $g^+_R$) yield that
	\begin{equation}\label{coupled:ineq:chi_plus_varphi:1}
		\begin{split}
			\vartheta(\{\star,\star\})&\int_{(\extK)\times (\extS)}\varphi_\varepsilon(y_1,y_2)\chi^+(t,\star,\star,\vartheta,d(y_1,y_2))\\ \leq  \int_{\mathcal{K}\times \mathcal{S}}&[C_g\varphi_\varepsilon(y_1,y_2)+C_{Lg}\sqrt{b^2+\varepsilon^2} (\|y_1-y_2\|+RW_1(\mu_1,\mu_2))]\vartheta(d(y_1,y_2)).
		\end{split}
	\end{equation}
	Further, we have that 
	\[\begin{split}
		\vartheta(\{\star,&\star\})\int_{(\extK)\times (\extS)}\varphi_\varepsilon(\star,\star)\chi^+(t,\star,\star,\vartheta,d(y_1,y_2))\\&=\varepsilon\int_{\mathcal{K}\times \mathcal{S}}((g_R^+(t,y_1,\mu_1)\vee g_R^+(t,y_2,\mu_2))-(g_R^+(t,y_1,\mu_1)\wedge g_R^+(t,y_2,\mu_2)))\\&{}\hspace{273pt}\vartheta(d(y_1,y_2))\\&\leq  C_g\varepsilon.
	\end{split}
	\]
	This and~\eqref{coupled:ineq:chi_plus_varphi:1} imply the statement of the lemma.
\end{proof}

\begin{proof}[Proof of Theorem~\ref{th:approx}]
	Let $\vartheta(\cdot):[0,T]\rightarrow\mathcal{P}^1((\extK)\times(\extS))$ satisfy the following conditions:
	\begin{itemize}
		\item $\vartheta(0)$ is an optimal plan between $\mu_{0,1}\triangleq R^{-1}\mext{m_0}{R}$ and $\mu_{0,2}\triangleq R^{-1}\mext{(\mathscr{I}(\mu_0))}{R}$;
		\item $\vartheta(\cdot)$ is a solution to~\eqref{coupled:eq:Lambda}.
	\end{itemize} 
	Notice that 
	\begin{equation}\label{coupled:ineq:W}
		W_{1}(\operatorname{p}^1\sharp\vartheta(t),\operatorname{p}^2\sharp\vartheta(t))\leq \int_{(\extK)\times (\extS)}\varphi_\varepsilon(x_1,x_2)\vartheta(t,d(x_1,x_2)).
	\end{equation} 
	Therefore, integrating the estimates proved in Lemmas~\ref{lm:dist_f_Q}--\ref{lm:chi_plus}, we conclude that 
	\[\begin{split}
		\int_{(\extK)\times (\extS)}\varphi_\varepsilon(x_1,x_2)&\vartheta(t,d(x_1,x_2))\\\leq \int_{(\extK)\times (\extS)}&\varphi_\varepsilon(x_1,x_2)\vartheta(0,d(x_1,x_2)) +C_{15}\varepsilon\\+C_{16}(1&+R)\int_0^t\int_{(\extK)\times (\extS)}\varphi_\varepsilon(x_1,x_2)\vartheta(s,d(x_1,x_2))ds.\end{split}\] Here $C_{15}$ and $C_{16}$ are constants determined only by functions $f$ and $g$. Using the Gronwall's inequality, we arrive at the estimate
	\[\begin{split}
		\int_{(\extK)\times (\extS)}\varphi_\varepsilon(x_1,x_2)&\vartheta(t,d(x_1,x_2))\\\leq \Bigg[\int_{(\extK)\times (\extS)}&\varphi_\varepsilon(x_1,x_2)\vartheta(0,d(x_1,x_2))+C_{15}\varepsilon\Bigg]e^{C_{16}(1+R)T}.
	\end{split}	
	\] Further, recall that $\varphi_\varepsilon(x_1,x_2)=\sqrt{\|x_1-x_2\|^2+\varepsilon^2}\leq \|x_1-x_2\|+\varepsilon$. Using this and~\eqref{coupled:ineq:W}, we conclude that 
	\[W_1(\operatorname{p}^1\sharp\vartheta(t),\operatorname{p}^2\sharp\vartheta(t))\leq \big[W_1(\operatorname{p}^1\sharp\vartheta(t),\operatorname{p}^2\sharp\vartheta(t))+(C_{15}+1)\varepsilon\big]e^{C_{16}(1+R)T}.
	\] The conclusion of the theorem follows from the definition of the flow $\vartheta(\cdot)$, equality~\eqref{intro:W_p_b} and Propositions~\ref{th:restriction_kynetic},~\ref{prop:generator_finite},~\ref{prop:generator_coupling}.
\end{proof}





\begin{appendices}
	
	\section{Uniqueness result for linear balance equation}\label{appendix:lm_unique}
	In this section we proof Lemma \ref{lm:unique} that is used in the proof of Theorem \ref{th:superposition}.
	
	\begin{proof}[Proof of Lemma \ref{lm:unique}]
		We will use the method borrowed from \cite[\S 8.1]{Ambrosio}. First, let $m_1(\cdot)$, $m_2(\cdot)$ be two solutions of~\eqref{super:eq:PDE}. Thus, the flow of signed measures $m_S(\cdot)$ defined by the rule
		\[m_S(t)=m_1(t)-m_2(t)\] also satisfies~\eqref{super:eq:PDE} in the weak sense with $m_0=0$.
		
		If $(s,y)\in [0,T]\times\rd$, then denote by $x(\cdot,s,y)$, the solution of the initial value problem
		\[\frac{d}{dt}x(t,s,y)=v(t,x(t,s,y)),\ \ x(s,s,y)=y.\] Further, let $w(\cdot,y)$ satisfy
		\[\frac{d}{dt}w(t,y)=-z(t,x(t,0,y))w(t,y),\ \ w(0,y)=1.\]  Notice that $w(t,y)>0$ for all $t\in [0,T]$. Additionally, set
		\[\zeta(s,y)=w(s,x(0,s,y)).\] This implies that 
		\[\frac{d}{dt}\zeta(t,x(t,s,y))=-z(t,x(t,s,y))\zeta(t,x(t,s,y))\]
		
		Choose an arbitrary function $\psi\in C_c^1([0,T]\times\rd)$ and set
		\[\xi(s,y)\triangleq -\int_s^T\psi(t,x(t,s,y))(\zeta(t,x(t,s,y)))^{-1}dt.\] 
		
		Now let us introduce the function $\phi(s,y)\triangleq \xi(s,y)\zeta(s,y)$. Notice that $\phi(T,y)\equiv 0$. Moreover, if we define $y'\triangleq x(0,s,y)$, then 
		\[\phi(s,y)=\phi(s,x(s,0,y')).\] Hence,
		\[\begin{split}
			\partial_s\phi(&s,y)+\langle\nabla\phi(s,y),v(s,y)\rangle= \frac{d}{ds}\phi(s,x(s,0,y'))\\&=\psi(s,x(s,0,y'))-\xi(s,x(s,0,y'))\zeta(s,x(s,0,y')) z(s,x(s,0,y')).\end{split}\] Due to  the definitions of $y'$ and the function $\phi$, we have that 
		\[\partial_s\phi(s,y)+\langle\nabla\phi(s,y),v(s,y)\rangle=\psi(s,y)-\phi(s,y)z(s,y).\]
		Therefore, using the equalities $m_S(0)=0$, $\phi(T,\cdot)\equiv 0$ and the fact that $m_S(\cdot)$ satisfies~\eqref{super:eq:PDE}, we conclude that
		\[\begin{split}
			0&=\int_0^T\int_{\rd}[\partial_t\phi(t,x)+\langle\nabla\phi(t,x),v(t,x)\rangle+\phi(t,x)z(t,x)]m_S(t,dx)dt\\&=
			\int_0^T\int_{\rd}\psi(t,x)m_S(t,dx)dt.
		\end{split}\] This means that, for every $\psi\in C_c^1([0,T]\times \rd)$,
		\[\int_0^T\int_{\rd}\psi(t,x)m_S(t,dx)dt=0.\]
		Thus, $m_1(t)=m_2(t)$ for a.e. $t\in [0,T]$.
	\end{proof}
	
	\begin{remark} Using the convolution technique, one can extend the result of Lemma \ref{lm:unique} to the case where $v$ is only continuous.
	\end{remark}

	\section{Distances between measures on a finite set}\label{appendix:distance}
	The aim of this appendix is to show the equivalence between distances $\|\beta^1-\beta^2\|$ and $\mathcal{W}_{1,b}(\mathscr{I}(\beta^1),\mathscr{I}(\beta^2))$ stated in Section~\ref{sect:finite}.
	\begin{proof}[Proof of Proposition~\ref{prop:metric_equivalence}]
		First, given $\beta^1_{\mathcal{S}},\beta^2_{\mathcal{S}}\in\mathscr{l}_1^+(\mathcal{S})$, we denote by $\mathcal{A}(\beta^1_{\mathcal{S}},\beta^2_{\mathcal{S}})$ the set of  matrices $(a_{{x},{y}})_{{x},{y}\in\mathcal{S}\cup\{\star\}}$ with nonnegative entries such that, for each ${x}\in\mathcal{S}$ and $y\in\mathcal{S}$,
		\[\sum_{{y'}\in\mathcal{S}\cup\{\star\}}a_{{x},{y'}}=\beta^1_{{x}},\ \ \sum_{{x'}\in\mathcal{S}\cup\{\star\}}a_{{x'},{y}}=\beta^2_{{y}}. \]

		From the definition of the metric $\mathcal{W}_{1,b}$ and the isomorphism $\beta\mapsto\mathscr{I}(\beta_{\mathcal{S}})$, it follows that 
		\begin{equation*}\label{appendix_A:equality:W} 
			\mathcal{W}_{1,b}(\mathscr{I}(\beta^1_{\mathcal{S}}), \mathscr{I}(\beta^2_{\mathcal{S}})) =\inf\Bigg\{\sum_{{x},{y}\in\mathcal{S}\cup\{\star\}} \|{x}-{y}\|a_{{x},{y}}:(a_{{x},{y}})_{{x},{y}\in \mathcal{S}\cup\{\star\}}\in\mathcal{A}(\beta^1,\beta^2) \Bigg\}.
		\end{equation*} 
		Recall  that, by convention, for $x\in\mathcal{K}$, $\|x-\star\|=\|\star-x\|=b$, while $\|\star-\star\|=0$.
		
		Let us now prove the first inequality in \eqref{finite:ineq:W_metric}. To this end, we consider the  element $(a'_{{x},{y}})_{{x},{y}\in\mathcal{S}\cup\{\star\}}\subset \mathcal{A}(\beta_{\mathcal{S}}^1,\beta_{\mathcal{S}}^2)$ defined by the rule:
		\begin{equation*}
			a'_{{x},{y}}\triangleq \left\{
			\begin{array}{ll}
				\beta_{{x}}^1\wedge\beta_{{x}}^2, & {x}={y}\in\mathcal{S},\\
				0, & {x}\neq {y},\, {x},{y}\in\mathcal{S},\\
				\beta_{{x}}^1-(\beta_{{x}}^1\wedge\beta_{{x}}^2), & {x}\in\mathcal{S},\, {y}=\star,\\
				\beta_{{y}}^2-(\beta_{{y}}^1\wedge\beta_{{y}}^2), & {x}=\star,\, {y}\in \mathcal{S}.
			\end{array}
			\right.
		\end{equation*} We have that 
		\[\begin{split}
			\mathcal{W}_{1,b}(\mathscr{I}(&\beta^1_{\mathcal{S}}), \mathscr{I}(\beta^2_{\mathcal{S}}))\leq \sum_{{x},{y}\in\mathcal{S}\cup\{\star\}}\|{x}-{y}\|a'_{{x},{y}}\\&= \sum_{{x}\in\mathcal{S}}b[(\beta^1_{{x}}-(\beta_{{x}}^1\wedge\beta_{{x}}^2))+(\beta^2_{{x}}-(\beta_{{x}}^1\wedge\beta_{{x}}^2))]\\&= b\sum_{{x}\in\mathcal{S}}[(\beta_{{x}}^1\vee\beta_{{x}}^2)-(\beta_{{x}}^1\wedge\beta_{{x}}^2)]=b\sum_{{x}\in\mathcal{S}}|\beta_{{x}}^1-\beta_{{x}}^2|=b\|\beta_{\mathcal{S}}^1-\beta_{\mathcal{S}}^2\|_1 .\end{split}\] This gives the first inequality in \eqref{finite:ineq:W_metric}.
		
		Further, for each $(a_{{x},{y}})_{{x},{y}\in\mathcal{S}}\in \mathcal{A}(\beta_{\mathcal{S}}^1,\beta_{\mathcal{S}}^2)$, we have that 
		\[\begin{split}
			\sum_{{x},{y}\in\mathcal{S}\cup\{\star\}}\|x-y\|a_{{x},{y}}&\geq d(\mathcal{S}) \sum_{{x}\in\mathcal{S}\cup\{\star\}} \sum_{{y}\in\mathcal{S}\cup\{\star\},{y}\neq{x}} a_{{x},{y}}\\&\geq d(\mathcal{S})\sum_{{x}\in\mathcal{S}}|\beta^1_{{x}}-\beta^2_{{x}}| =d(\mathcal{S})\|\beta^1-\beta^2\|.
		\end{split}\] Taking the minimum w.r.t. elements of $\mathcal{A}(\beta^1,\beta^2)$, we arrive at the second inequality in~\eqref{finite:ineq:W_metric}.
		
	\end{proof}
	
\end{appendices}

\bibliography{source_sink}

\begin{thebibliography}{10}

\bibitem{Albi2019}
G.~Albi, M.~Bongini, F.~Rossi, and F.~Solombrino.
\newblock Leader formation with mean-field birth and death models.
\newblock {\em Math. Models Methods Appl. Sci.}, 29(04):633--679, 2019.

\bibitem{Infinite_dimensional}
C.~D. Aliprantis and K.~C. Border.
\newblock {\em Infinite dimensional analysis. {A} hitchhiker's guide}.
\newblock Springer, Berlin, 3rd edition, 2006.

\bibitem{Ambrosio}
L.~Ambrosio and G.~Crippa.
\newblock Existence, uniqueness, stability and differentiability properties of
  the flow associated to weakly differentiable vector fields.
\newblock In {\em Transport equations and multi-D hyperbolic conservation
  laws}, pages 3--57. Springer, Berlin, 2008.

\bibitem{Ambrosio_Crippa}
L.~Ambrosio and G.~Crippa.
\newblock Continuity equations and ode flows with non-smooth velocity.
\newblock {\em Proc. Roy. Soc. Edinburgh Sect., A}, 144:1191–1244, 2014.

\bibitem{Ambrosio2008b}
L.~Ambrosio, N.~Gigli, and G.~Savar{\'e}.
\newblock {\em Gradient flows in metric spaces and in the space of probability
  measures}.
\newblock Birkh{\"a}user, Basel, 2nd edition, 2008.

\bibitem{Averboukh2016}
Y.~Averboukh.
\newblock Approximate solutions of continuous-time stochastic games.
\newblock {\em SIAM J. Control Optim.}, 54:2629--2649, 2016.

\bibitem{Averboukh2021}
Y.~Averboukh.
\newblock Lattice approximations of the first-order mean field type
  differential games.
\newblock {\em Nonlinear Differ. Equ. Appl.}, 28, 2021.

\bibitem{Ayi2021}
N.~Ayi and N.~Pouradier~Duteil.
\newblock Mean-field and graph limits for collective dynamics models with
  time-varying weights.
\newblock {\em J. Differ. Equ.}, 299:65--110, 2021.

\bibitem{Bayen2022}
A.~Bayen, J.~Friedrich, A.~Keimer, L.~Pflug, and T.~Veeravalli.
\newblock Modeling multilane traffic with moving obstacles by nonlocal balance
  laws.
\newblock {\em SIAM J. Appl. Dyn. Syst.}, 21(2):1495--1538, 2022.

\bibitem{Bonicatto2018}
P.~Bonicatto and N.~A. Gusev.
\newblock Superposition principle for the continuity equation in a bounded
  domain.
\newblock {\em J. Phys. Conf. Ser.}, 990(1):012002, 2018.

\bibitem{Bredies2022}
K.~Bredies, M.~Carioni, and S.~Fanzon.
\newblock A superposition principle for the inhomogeneous continuity equation
  with {H}ellinger--{K}antorovich-regular coefficients.
\newblock {\em Commun. Partial. Differ. Equ.}, 47(10):2023--2069, 2022.

\bibitem{Colombo2015}
R.~M. Colombo and F.~Marcellini.
\newblock Nonlocal systems of balance laws in several space dimensions with
  applications to laser technology.
\newblock {\em J. Differ. Equ.}, 259:6749--6773, 2015.

\bibitem{Crippa_existence}
G.~Crippa and M.~L\'{e}cureux-Mercier.
\newblock Existence and uniqueness of measure solutions for a system of
  continuity equations with non-local flow.
\newblock {\em Nonlinear Differ. Equ. Appl.}, 20:523--537, 2013.

\bibitem{Duteil2022}
N.~P. Duteil.
\newblock Mean-field limit of collective dynamics with time-varying weights.
\newblock {\em Netw. Heterog. Media}, 17(2):129--161, 2022.

\bibitem{Keimer2018}
A.~Keimer, G.~Leugering, and T.~Sarkar.
\newblock Analysis of a system of nonlocal balance laws with weighted work in
  progress.
\newblock {\em J. Hyperbolic Differ. Equ.}, 15:375--406, 2018.

\bibitem{KEIMER2023183}
A.~Keimer and L.~Pflug.
\newblock Chapter 6 - nonlocal balance laws – an overview over recent
  results.
\newblock In E.~Trélat and E.~Zuazua, editors, {\em Numerical Control: Part
  B}, volume~24 of {\em Handbook of Numerical Analysis}, pages 183--216.
  Elsevier, Amsterdam, The Netherlands, 2023.

\bibitem{Khlopin2022}
D.~V. Khlopin.
\newblock Differential game with discrete stopping time.
\newblock {\em Autom. Remote. Control}, 83:649--672, 2022.

\bibitem{Kolokoltsov_DE}
V.~Kolokoltsov.
\newblock {\em Differential Equations on Measures and Functional Spaces}.
\newblock Birkh{\"{a}}user, Basel, 2019.

\bibitem{Kolokoltsov_Markov}
V.~N. Kolokoltsov.
\newblock {\em Markov Processes, Semigroups and Generators}.
\newblock De Gruyter, Berlin, Germany, 2010.

\bibitem{Maniglia2007}
S.~Maniglia.
\newblock Probabilistic representation and uniqueness results for
  measure-valued solutions of transport equations.
\newblock {\em J. Math. Pures Appl.}, 87(6):601--626, 2007.

\bibitem{McQuade2019}
S.~McQuade, B.~Piccoli, and N.~P. Duteil.
\newblock Social dynamics models with time-varying influence.
\newblock {\em Math. Models Methods Appl. Sci.}, 29(4):681--716, 2019.

\bibitem{Piccoli2021}
B.~Piccoli and N.~P. Duteil.
\newblock Control of collective dynamics with time-varying weights.
\newblock In {\em Recent Advances in Kinetic Equations and Applications}, pages
  289--308. Springer, 2021.

\bibitem{Piccoli2014}
B.~Piccoli and F.~Rossi.
\newblock Generalized {{W}asserstein} distance and its application to transport
  equations with source.
\newblock {\em Arch. Ration. Mech. Anal.}, 211(1):335--358, 2014.

\bibitem{Piccoli2016}
B.~Piccoli and F.~Rossi.
\newblock On properties of the generalized {W}asserstein distance.
\newblock {\em Arch. Ration. Mech. Anal.}, 222:1339--1365, 2016.

\bibitem{Piccoli2018}
B.~Piccoli and F.~Rossi.
\newblock Measure-theoretic models for crowd dynamics.
\newblock In {\em Crowd dynamics, Volume 1. Theory, models, and safety
  problems}, pages 137--165. Birkh{\"a}user, Cham, 2018.

\bibitem{Piccoli2019}
B.~Piccoli, F.~Rossi, and M.~Tournus.
\newblock A {W}asserstein norm for signed measures, with application to
  nonlocal transport equation with source term.
\newblock {\em Preprint at ArXiv:1910.05105}, 2019.

\bibitem{Pogodaev2022}
N.~I. Pogodaev and M.~V. Staritsyn.
\newblock Nonlocal balance equations with parameters in the space of signed
  measures.
\newblock {\em Mat.Sb.}, 213:63--87, 2022.

\bibitem{Stepanov2017}
E.~Stepanov and D.~Trevisan.
\newblock Three superposition principles: currents, continuity equations and
  curves of measures.
\newblock {\em J. Funct. Anal.}, 272(3):1044--1103, 2017.

\end{thebibliography}

\end{document}